\documentclass{article}
\usepackage{amsthm,amsfonts, amsbsy, amssymb,amsmath,graphicx}
\usepackage{graphics}

\newtheorem{thm}{Theorem}
\newtheorem{lm}{Lemma}
\newtheorem{re}{Remark}

\newtheorem{crl}{Corollary}

\author{Vassily Olegovich Manturov}

\date{}

\title{Additional Gradings in Khovanov homology}

\begin{document}

\maketitle

\begin{abstract}

The main goal of the present paper is to construct new invariants of
knots with additional structure by adding new gradings to the
Khovanov complex. The ideas given below work in the case of virtual
knots, closed braids and some other cases of knots with additional
structure. The source of our additional grading may be topological
or combinatorial; it is axiomatised for many partial cases. As a
byproduct, this leads to a complex which in some cases coincides (up
to grading renormalisation) with the usual Khovanov complex and in
some other cases with the Lee-Rasmussen complex.

The grading we are going to construct behaves well with respect to
some generalisations of the Khovanov homology, e.g., Frobenius
extensions. These new homology theories give sharper estimates for
some knot cha\-rac\-te\-ris\-tics, such as minimal crossing number,
atom genus, slice genus, etc.

Our gradings generate a natural filtration on the usual Khovanov
complex. There exists a spectral sequence starting with our homology
and converging to the usual Khovanov homology.
\end{abstract}

\section{Introduction}

\newcommand{\skcrossr}{\raisebox{-0.25\height}{\includegraphics[width=0.5cm]{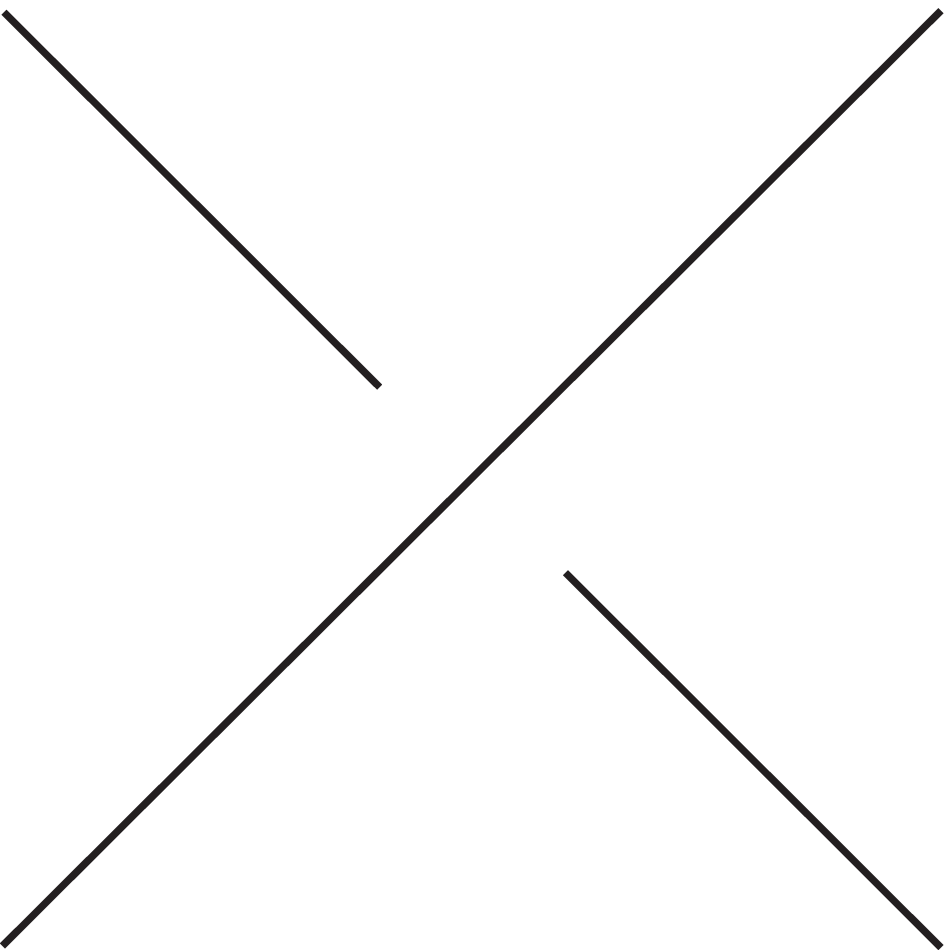}}}
\newcommand{\skcrossl}{\raisebox{-0.25\height}{\includegraphics[width=0.5cm]{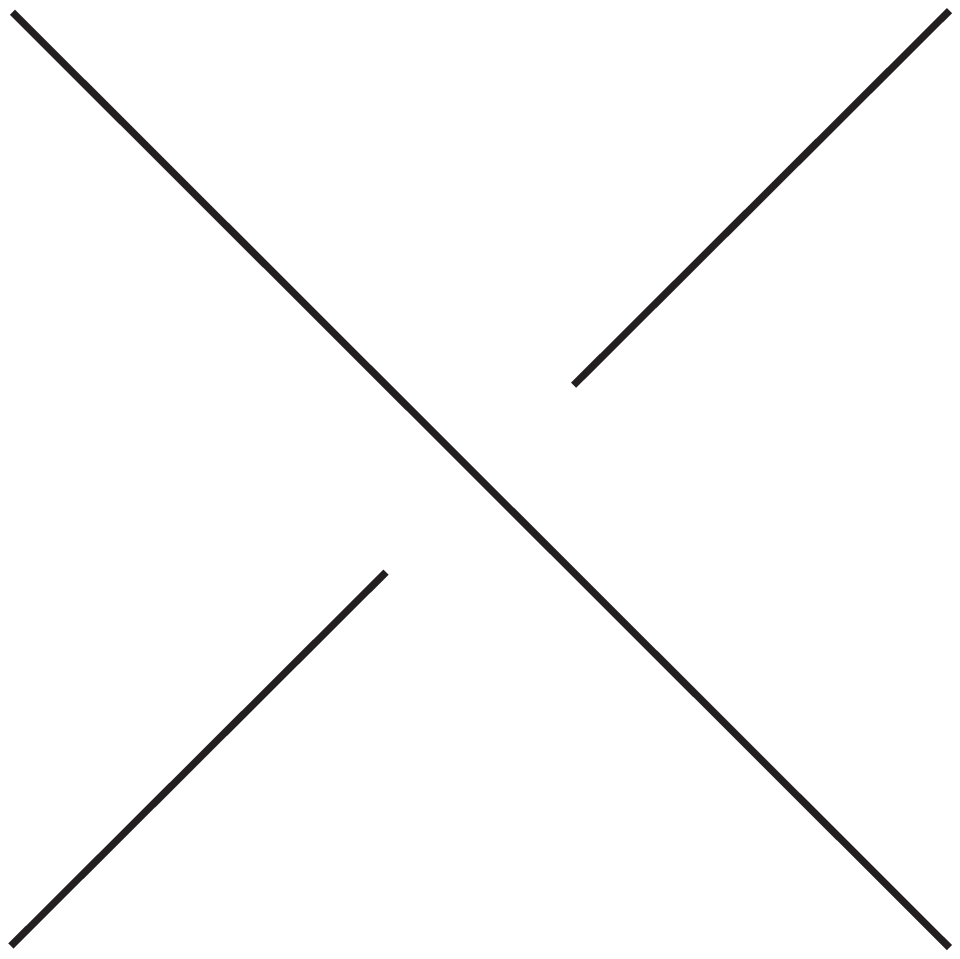}}}
\newcommand{\skkinkr}{\raisebox{-0.25\height}{\includegraphics[width=0.5cm]{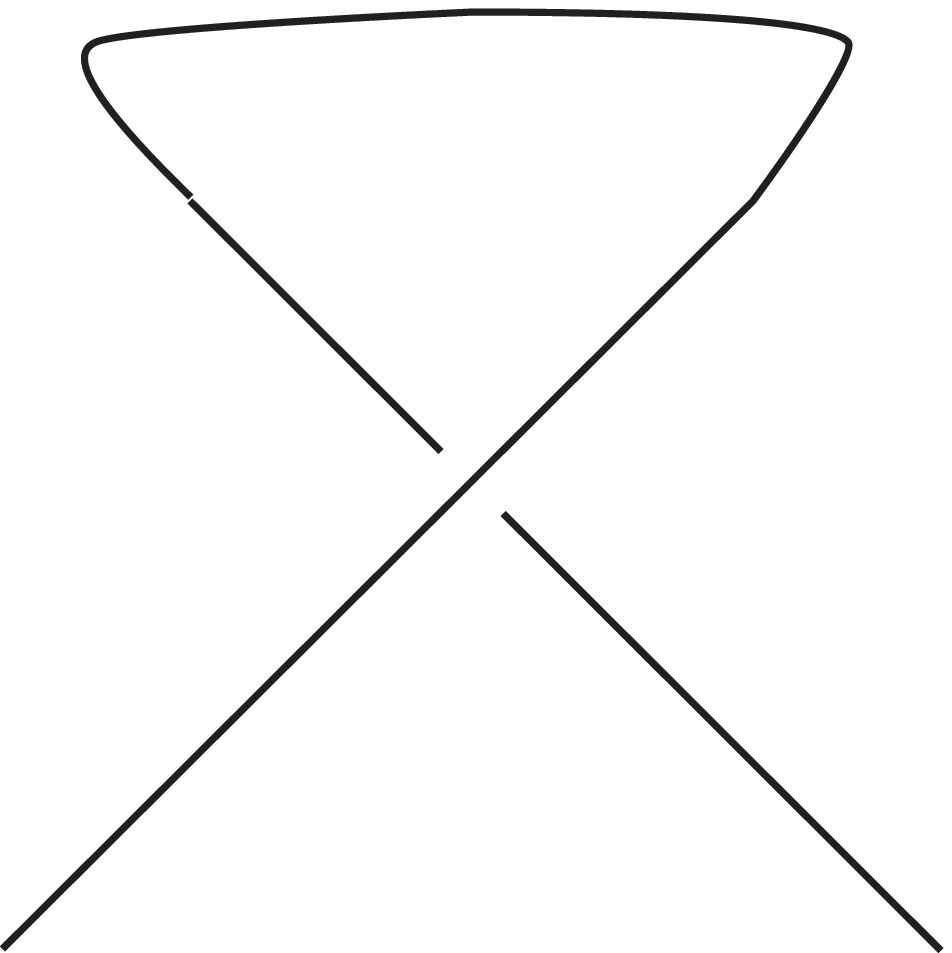}}}
\newcommand{\skkinkl}{\raisebox{-0.25\height}{\includegraphics[width=0.5cm]{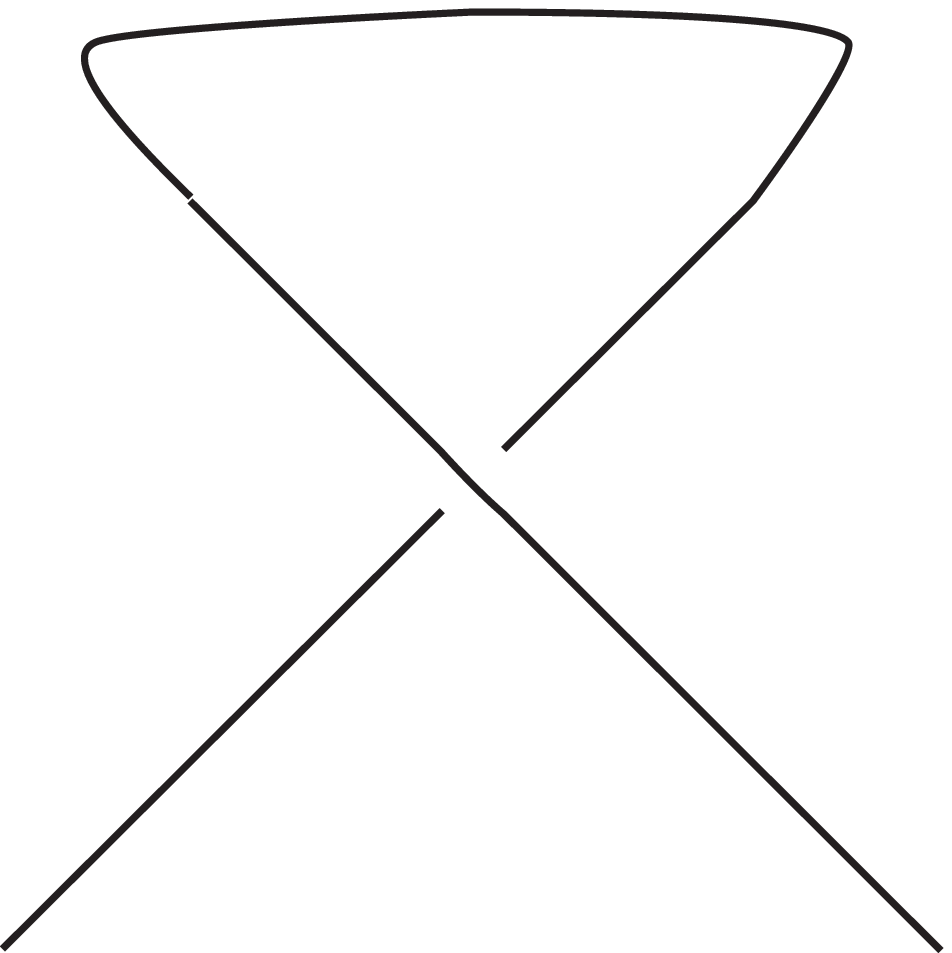}}}
\newcommand{\skroh}{\raisebox{-0.25\height}{\includegraphics[width=0.5cm]{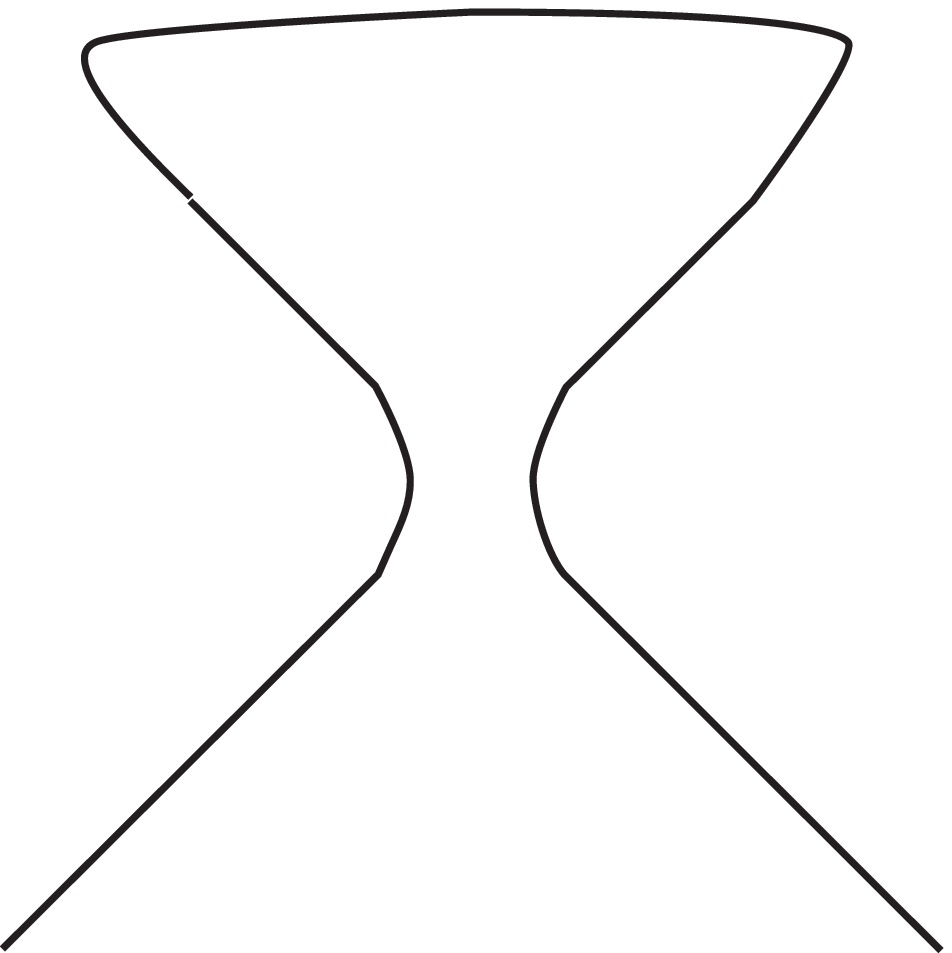}}}
\newcommand{\skrov}{\raisebox{-0.25\height}{\includegraphics[width=0.5cm]{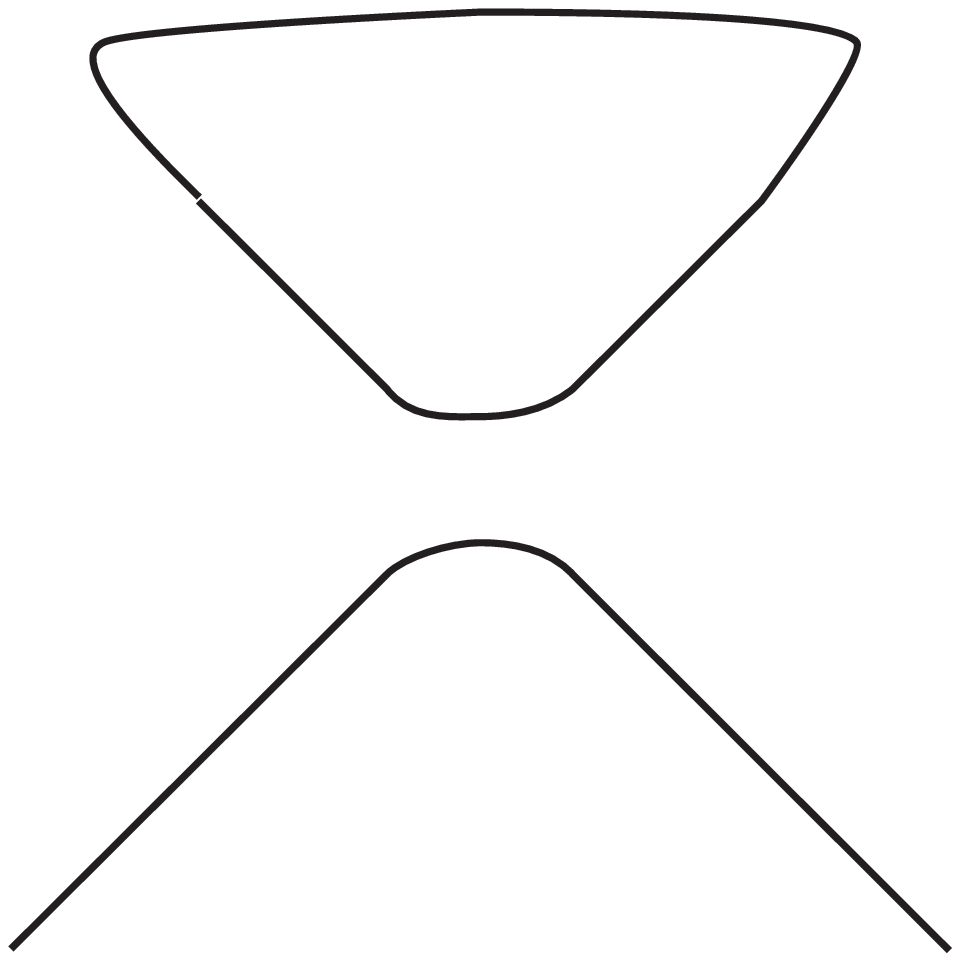}}}
\newcommand{\skrtwhh}{\raisebox{-0.25\height}{\includegraphics[width=0.5cm]{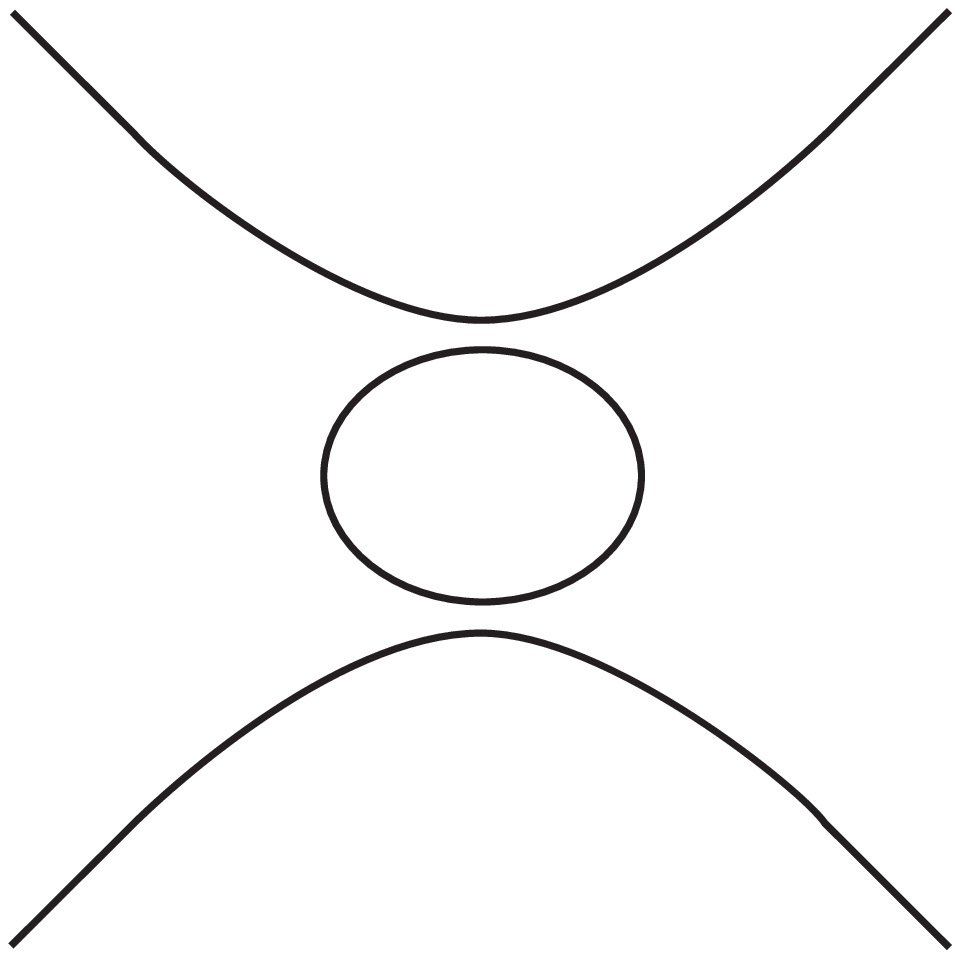}}}
\newcommand{\skrtwvh}{\raisebox{-0.25\height}{\includegraphics[width=0.5cm]{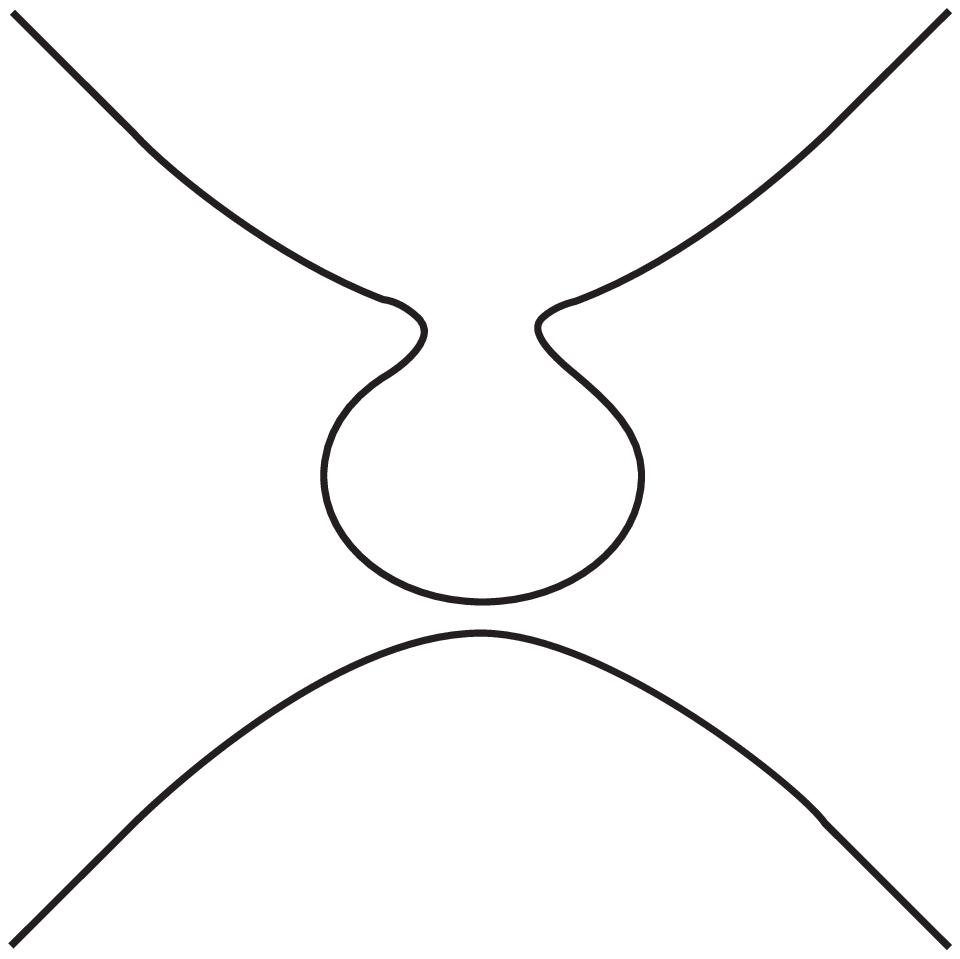}}}
\newcommand{\skrtwhv}{\raisebox{-0.25\height}{\includegraphics[width=0.5cm]{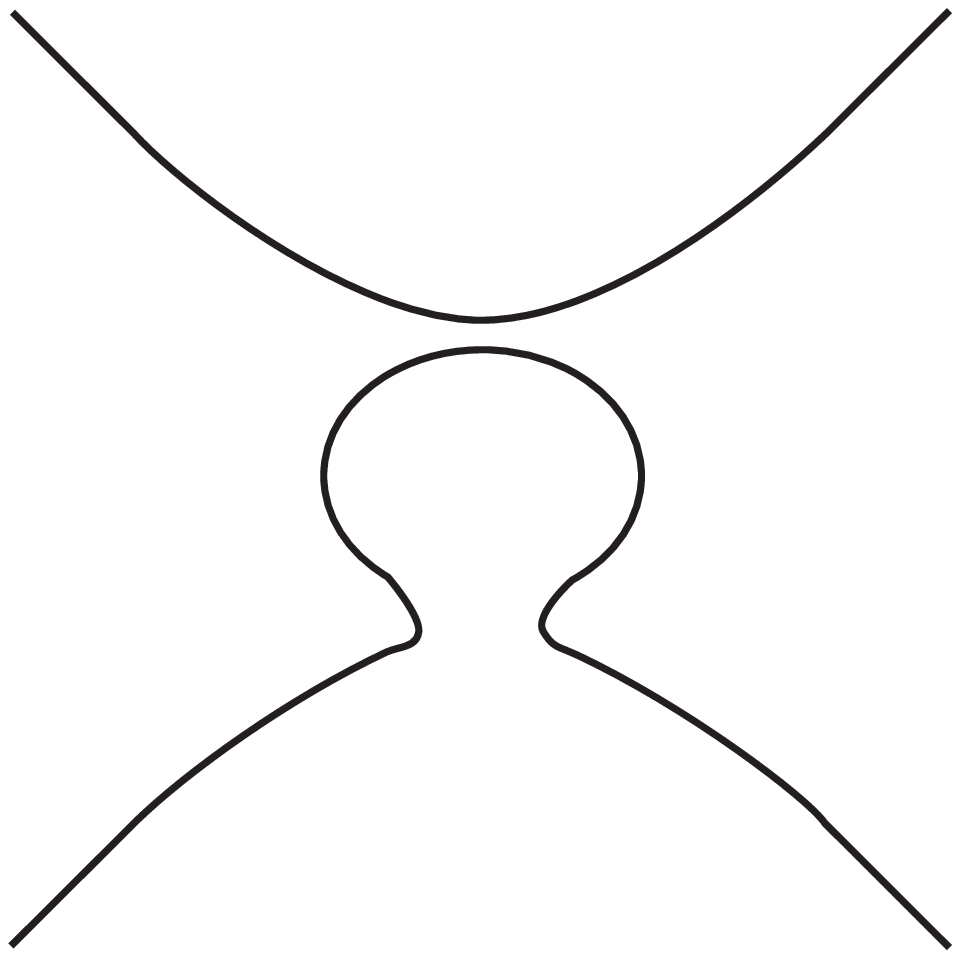}}}
\newcommand{\skrtwvv}{\raisebox{-0.25\height}{\includegraphics[width=0.5cm]{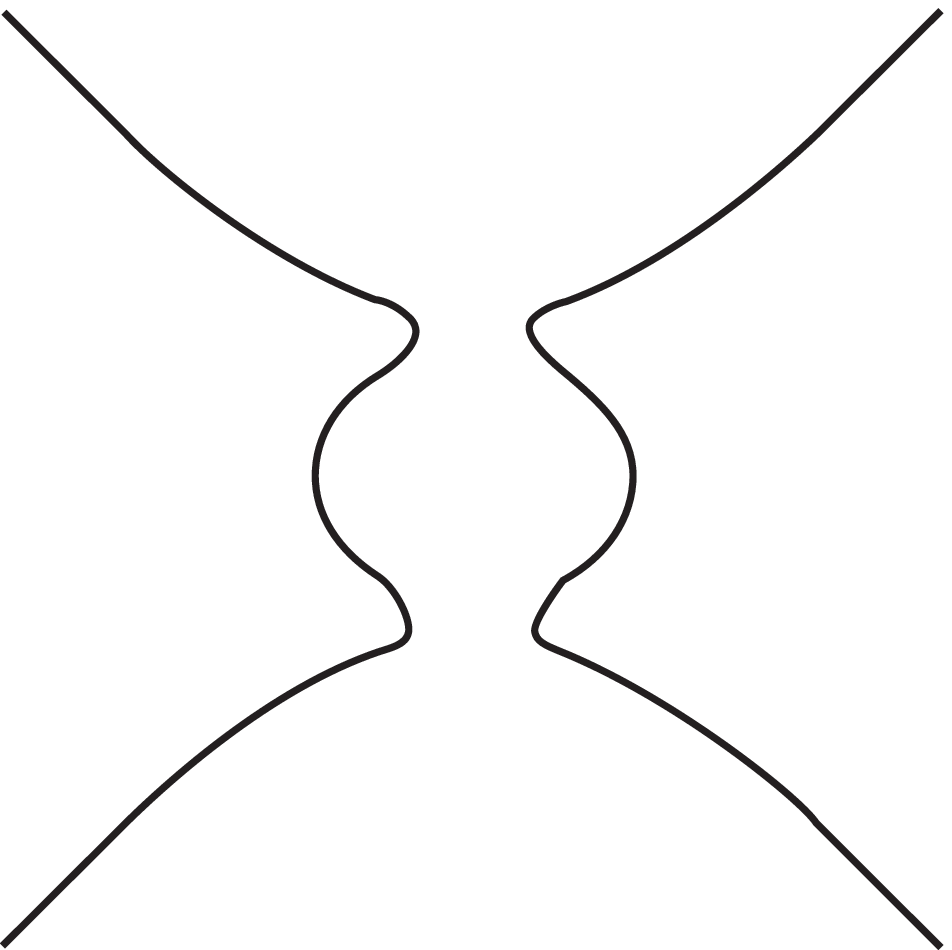}}}
\newcommand{\skcrv}{\raisebox{-0.25\height}{\includegraphics[width=0.5cm]{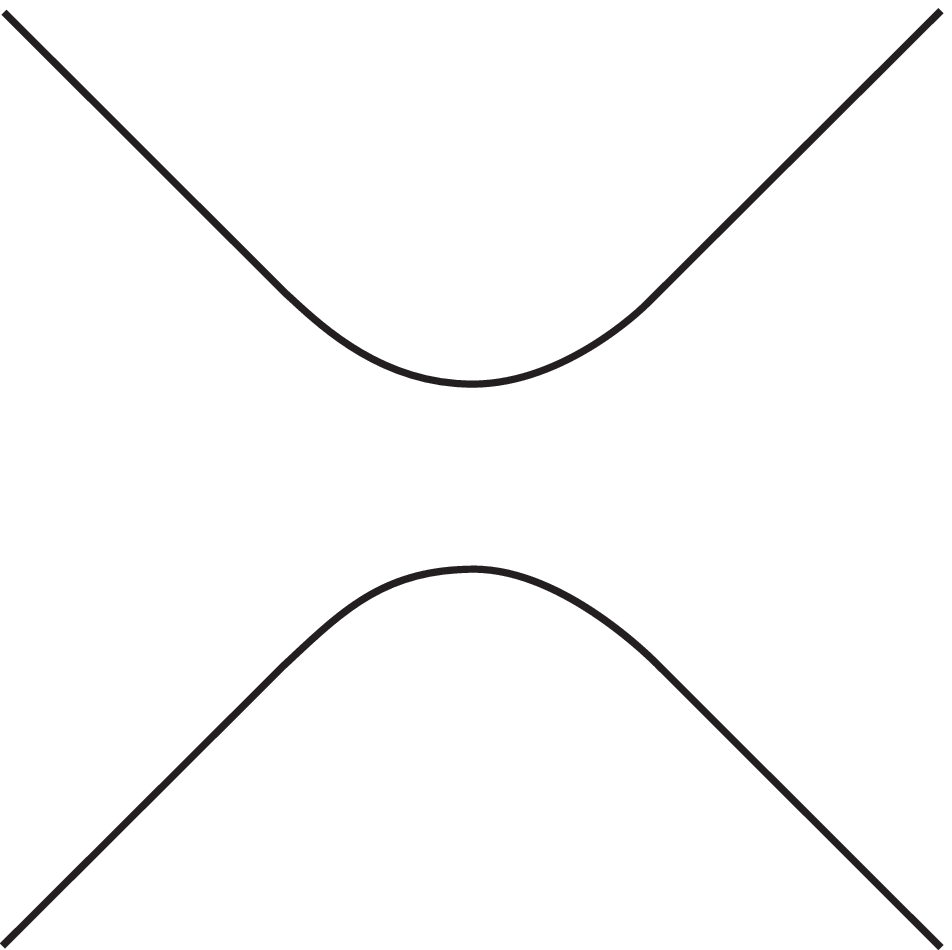}}}
\newcommand{\skcrh}{\raisebox{-0.25\height}{\includegraphics[width=0.5cm]{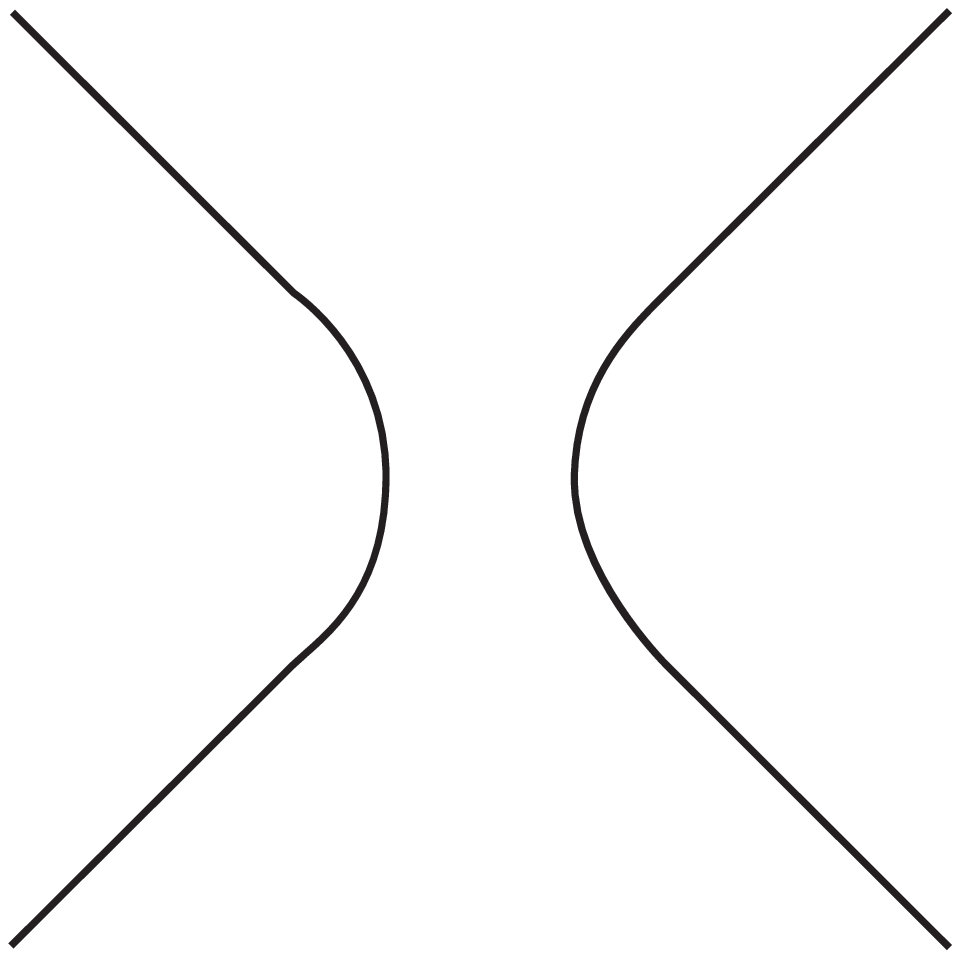}}}
\newcommand{\vcross}{\raisebox{-0.25\height}{\includegraphics[width=0.5cm]{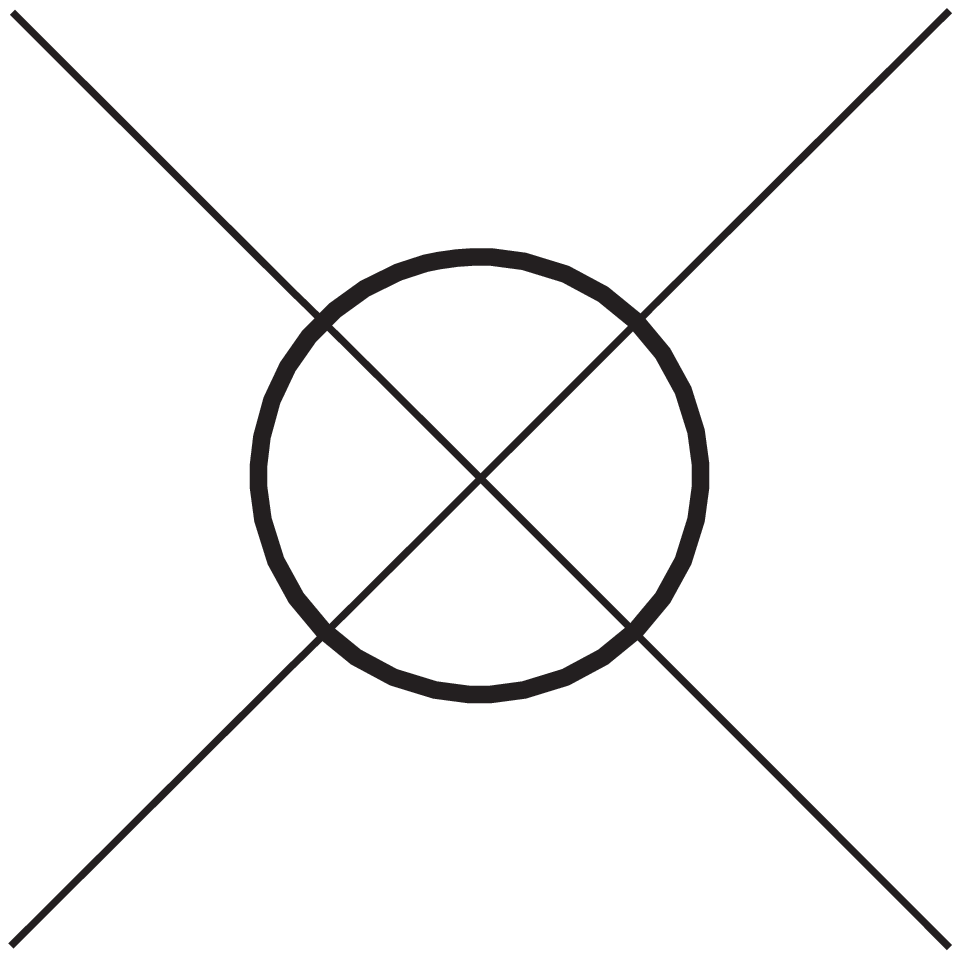}}}
\newcommand{\skcurl}{\raisebox{-0.25\height}{\includegraphics[width=0.5cm]{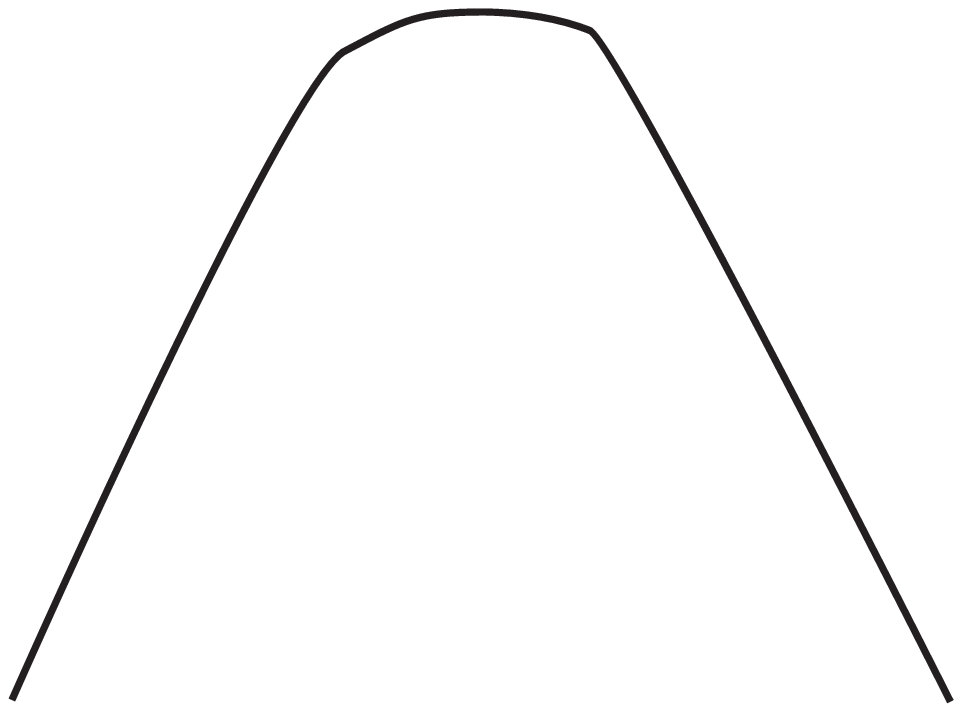}}}
\newcommand{\skkinkv}{\raisebox{-0.25\height}{\includegraphics[width=0.5cm]{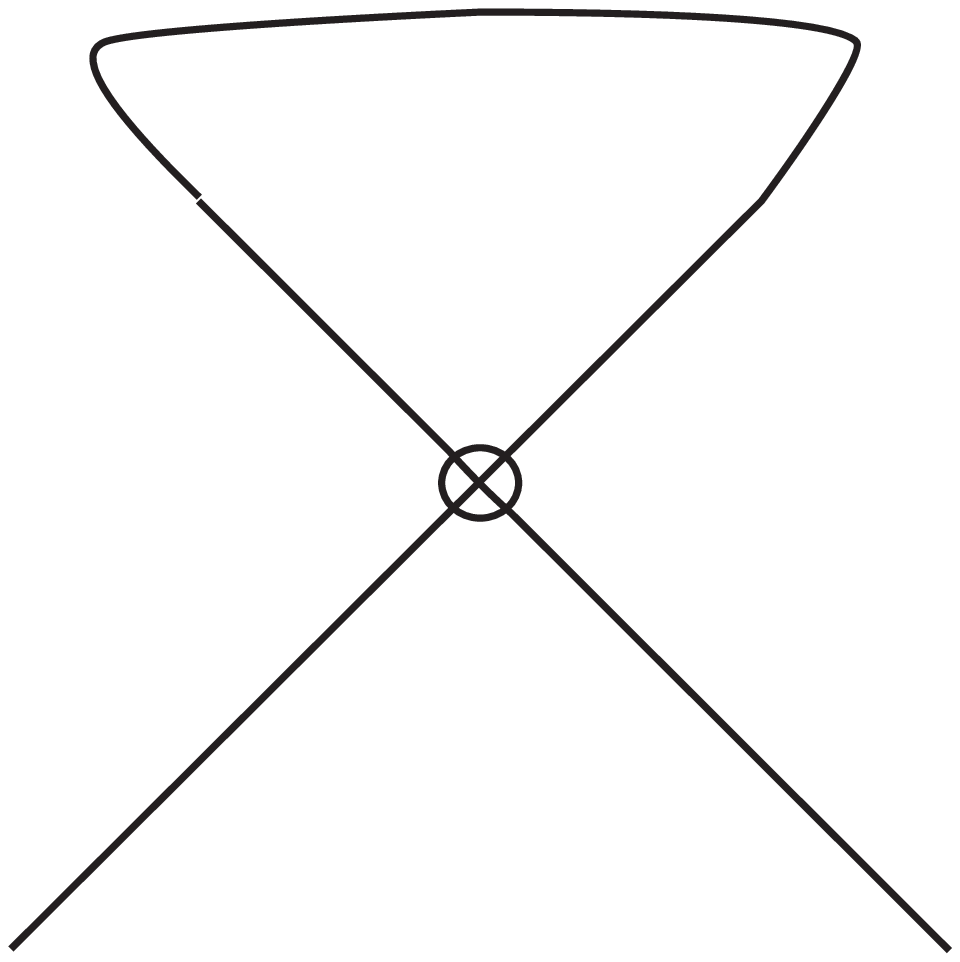}}}

In the last few years, the invention of link homology (Khovanov
homology, Ozsv\'ath-Szab\'o invariants, and also papers by
Rasmussen, Khovanov-Rozansky, Manolescu-Ozsv\'ath-Sarkar-Thurston
and others) brought many constructions from algebraic topology to
knot theory and low-dimensional topology.

Such theories take a representative of a low-dimensional diagram
(say, knot diagram or Heegaard diagram of a 3-manifold) and
associate a certain complex with this. The homology of this complex
is independent of the choice of repre\-sen\-tative, thus the
homology defines an invariant of knot (resp., 3-manifold, knot in a
manifold). Such algebraic complexes have different gradings, and
this allows one to construct filtrations and spectral sequences. The
behaviour of such spectral sequences is often closely connected to
some topological property of knots/3-manifolds. A nice example is
the work of Rasmussen \cite{Ras} estimating the Seifert genus from
Khovanov homology and giving a simple proof of Milnor's conjecture.
Another example is the work by K.Kawamura \cite{Kaw}, who sharpened
the Morton-Franks-Williams estimate for the braiding index.

There is also an approach to estimate the minimal crossing number,
see \cite{Minimal}.

We shall mainly concentrate on Khovanov homology. In a sequence of
recent papers, the author generalized Khovanov's theory from knots
in ${\bf R}^3$ to knots in arbitrary thickened $2$-surfaces (up to
stabilisation, giving virtual knots (by Kauffman, \cite{KaV}) or
twisted knots (by Bourgoin,\cite{Bou})).

Virtual knots, besides their ``knottedness'' also carry some
information about the topology of the underlying surface.

Thus, it would be quite natural to take into account some
topological data to introduce into Khovanov homology to make the
latter stronger. This idea was also used in the paper by Asaeda,
Przytycki, Sikora  \cite{APS}. We shall discuss the interaction
between the present work and the work \cite{APS} later.


The main idea goes as follows: Assume we have a well-defined complex
made out of some knot diagram. Consider the chain spaces ${\cal
C}(K)$ and the differential $\partial$. It turns out that in some
cases it is possible to introduce a new grading $gr$ that splits the
differential $\partial$ into two parts
$\partial=\partial'+\partial''$ in such a way that:

\begin{enumerate}

\item $\partial'$ preserves the new grading, whence $\partial''$
increases the new grading;

\item $(C,\partial')$ is a well-defined complex;

\item the homology of $(C,\partial')$ is invariant (under
Reidemeister moves);

\item there is a spectral sequence with $E^{1}=H(C,\partial)$
converging to the usual Khovanov homology (the latter differential
is taken with respect to $\partial$).

\end{enumerate}

The new gradings have a topological nature: they correspond to
cohomology classes.


This will guarantee that the complex is well defined. However, the
gradings may be of any other (say, combinatorial) nature; the only
thing we need is that for the Kauffman bracket states, there are two
sorts of circles which behave nicely with respect to the
Reidemeister moves.

The latter condition guarantees not only that the complex is well
defined (that is, $\partial'$ is indeed a differential) but also the
invariance under Reidemeister moves.

Varying this construction, one can construct further complexes with
dif\-fe\-ren\-ti\-als of type $\partial'+\lambda\partial''$, where
$\lambda$ can be a coefficient or some operator.

The outline of the present paper is the following. In the next
section, we define the Kauffman bracket, virtual knots, and Khovanov
homology (with arbitrary coefficients) for virtual links and
classical links (which actually constitute a proper part of virtual
links).

Section 3 will be devoted to our main example: categorifying the
Bourgoin invariant with the only one new grading corresponding to
the first Stiefel-Whitney class for oriented thickenings of
non-orientable surfaces.

The proof of the invariance theorem is given in section 4; it indeed
contains all ingredients for the proof of the main theorem to follow
in section 5, where we have multiple gradings of various types and
present more examples.

Section 5 also contains the axiomatics for these new gradings and
examples what they can be applied to: braids, cables, tangles, long
knots etc.

Section 6 devoted is to a generalisation of Khovanov's Frobenius
structure. From this point of view, one can think of Lee's homology
as a partial case of Khovanov's Frobenius theory as well as our new
theory. As a byproduct, we present yet another definition of the
Khovanov theory where the usual gradings are treated from our
``dotted grading viewpoint''.

In section 7, we focus on gradings and filtrations. We discuss the
Frobenius construction due to Khovanov, which is then followed by
spectral sequences, and Lee-Rasmussen invariants.

Section 8 is devoted to applications of the theory constructed and
gene\-ra\-li\-sations of some classical constructions in this
context

Section 9 is devoted to the discussion and open questions.

\subsection{Acknowledgments}

I am very grateful to O.Ya.Viro for various stimulating discussions.
During the writing process of that paper, I also had very useful
dis\-cus\-si\-ons with L.H.Kauff\-man, V.A.Vassiliev, M.Khovanov,
J.H.Przytycki, V.G.Turaev.

\section{Preliminaries: Virtual knots, Kauffman bracket, atoms, and Khovanov homology}

We think of knot diagrams\footnote{We refer both to knots and links
by using a generic term ``link''.} as a collection of (classical)
crossings on the plane somehow connected by arcs, see Fig.
\ref{clasknot}.

\begin{figure}
\centering\includegraphics[width=200pt]{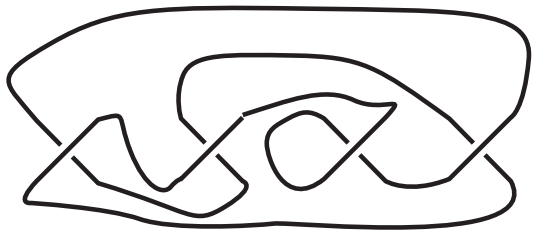} \caption{A
knot diagram} \label{clasknot}
\end{figure}

Knots are such diagrams modulo Reidemeister moves, but sometimes it
happens that for a given setup of crossings we are unable to connect
them by arcs in an appropriate way; this will lead to a diagram
called {\em virtual} with artefacts of the projection encircled, see
Fig. \ref{virtknot}.

\subsection{Atoms and Knots}

A four-valent planar graph $\Gamma$ generates a natural checkerboard
colouring of the plane by two colours (adjacent components of the
complement ${\bf R}^{2}\backslash \Gamma$ have different colours).

This construction perfectly describes the role played by {\em
alternating diagrams} of classical knots. Recall that a link diagram
is {\em alternating} if while walking along any component we
alternate over= and underpasses. Another definition of an
alternating link diagram sounds as follows: fix a checkerboard
colouring of the plane (one of the two possible colourings). Then,
for every vertex the colour of the region corresponding to the angle
swept by going from the overpass to the underpass in the
counterclockwise direction is the same.

Thus, planar graphs with natural colourings somehow correspond to
alternating diagrams of knots and links on the plane: starting with
a graph and a colouring, we may fix the rule for making crossings:
if two edges share a black angle, then the we decree the left one
(with respect to the clockwise direction) to form an overcrossing,
and the right one to be an undercrossing, see Fig. \ref{overunder}.
Thus, colouring a couple of two opposite angles corresponds to a
choice of a pair of opposite edges to form an overcrossing and vice
versa.

\begin{figure}
\centering\includegraphics[width=200pt]{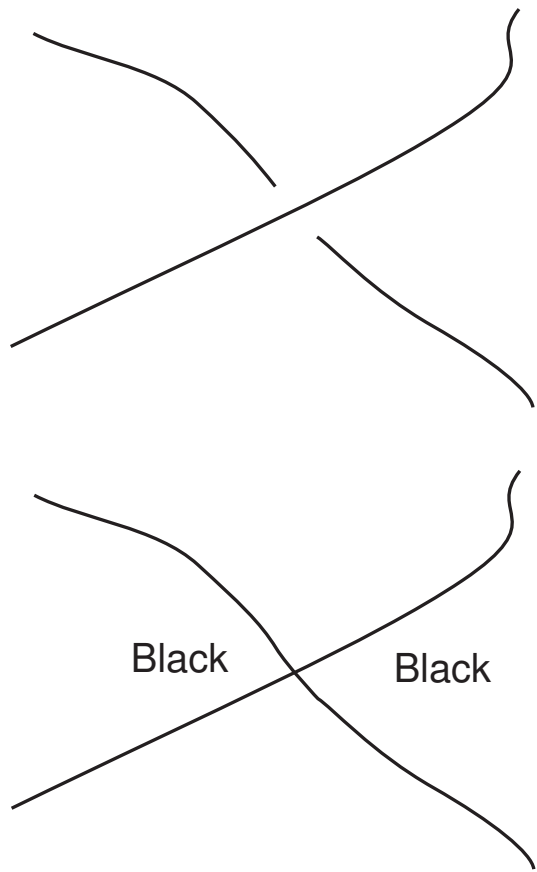} \caption{A
crossing corresponding to a vertex of an atom} \label{overunder}
\end{figure}

Now, if we take an arbitrary link diagram and try to establish the
colouring of angles according to the rule described above, we see
that generally it is impossible unless the initial diagram is
alternating: we can just get a region on the plane where colourings
at two adjacent angles disagree. So, alternating diagrams perfectly
match colourings of the $2$-sphere (think of $S^{2}$ as a one-point
compactification of ${\bf R}^{2}$). For an arbitrary link, we may
try to take colours and attach cells to them in a way that the
colours would agree, namely, the circuits for attaching two-cells
are chosen to be those rotating circuits, where we always turn
inside the angle of one colour.

This leads to the notion of {\em atom}. An {\em atom} is a pair
$(M,\Gamma)$ of a $2$-manifold $M$ and a graph $\Gamma$ embedded $M$
together with a colouring of $M\backslash \Gamma$ in a checkerboard
manner. Here $\Gamma$ is called the {\em frame} of the atom, whence
by {\em genus} (resp., {\em Euler characteristic}) of the atom we
mean that of the surface $M$.

Note that the atom genus is also called the {\em Turaev genus},
\cite{TuraevGenus}.

Certainly, such a colouring exists if and only if $\Gamma$
represents the trivial ${\bf Z}_{2}$ homology class in $M$.

Thus, gluing cells to some turning circuits on the diagram, we get
an atom, where the shadow of the knot plays the role of the frame.
Note that the structure of opposite half-edges on the plane
coincides with that on the surface of the atom.

Now, we see that atoms on the sphere are precisely those
corresponding to alternating link diagrams, whence non-alternating
link diagrams lead to atoms on surfaces of a higher genus.

In some sense, the genus of the atom is a measure of how far a link
diagram is from an alternating one, which leads to generalisations
of the celebrated Kauffman-Murasugi theorem, see \cite{MyBook} and
to some estimates concerning the Khovanov homology \cite{Minimal}.

Having an atom, we may try to embed its frame in ${\bf R}^{2}$ in
such a way that the structure of opposite half-edges at vertices is
preserved. Then we can take the ``black angle'' structure of the
atom to restore the crossings on the plane.

In \cite{AtomsAndKnots} it is proved that the link isotopy type does
not depend on the particular choice of embedding of the frame into
${\bf R}^{2}$ with the structure of opposite edges preserved. The
reason is that such embeddings are quite rigid.

The atoms whose frame is embeddable in the plane with opposite
half-edge structure preserved are called {\em height} or vertical.

However, not all atoms can be obtained from some classical knots.
Some abstract atoms may be quite complicated for its frame to be
embeddable into ${\bf R}^{2}$ with the opposite half-edges structure
preserved. However, if it is impossible to {\em immerse} a graph in
${\bf R}^{2}$, we may embed it by marking artifacts of the embedding
(we assume the embedding to be generic) by small circles.

A {\em virtual diagram} is a four-valent graph on the plane with two
types of crossings: classical $\skcrossr$ or $\skcrossl$ (for which
we mark which pair of opposite edges form an overpass) and virtual
$\vcross$ (which are just marked by a circled crossing).

A {\em virtual link} is an equivalence class of virtual diagrams
modulo generalised Reidemeister moves. The later consist of usual
Reidemeister moves and the detour move. The detour move removes an
arc virtually connecting some points $A$ and $B$ (that is, having no
classical crossings inside) restores another connection between $A$
and $B$ with several virtual intersections and self-intersections,
see Fig. \ref{detour}.

\begin{figure}
\centering\includegraphics[width=200pt]{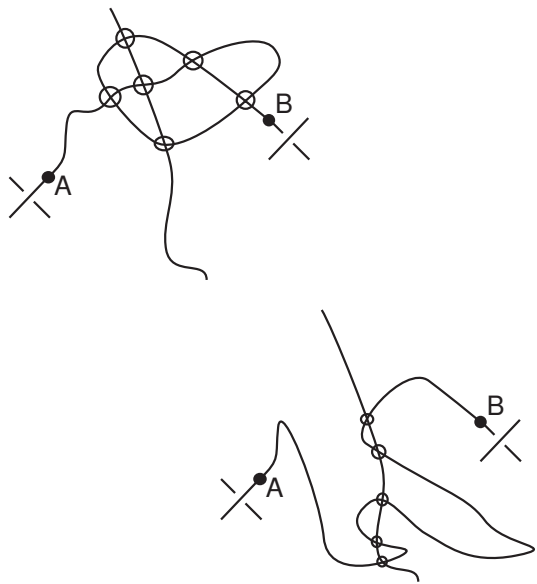} \caption{The
detour move} \label{detour}
\end{figure}

This move just means that it is inessential to indicate {\em which
curves connect classical crossings}, it is important only to know
how these crossings are paired.

\begin{figure}
\centering\includegraphics[width=200pt]{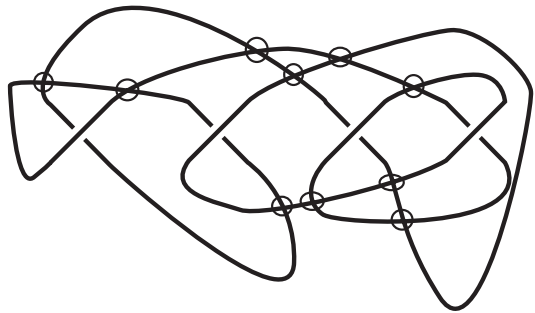} \caption{A
virtual diagram} \label{virtknot}
\end{figure}

Considering these diagrams modulo usual Reidemeister moves and the
detour moves (see ahead), we get what are called {\em virtual
knots}. The detour move is the move removing an arc (possibly, with
self-intersections) containing only virtual crossing, and adding
another arc connecting the same points elsewhere.

Virtual knots, being defined diagrammatically, have a topological
interpretation. They correspond to knots in thickened surfaces
$S_{g}\times I$ with fixed $I$-bundle structure (later we will also
talk about oriented thickenings of non-orientable surfaces) up to
stabilisations/destabilisations. Projecting  $S_{g}$ to ${\bf
R}^{2}$ (with the condition, however, that all neighbourhoods of
crossings are projected with respect to the orientation, we get from
a generic diagram on $S_g$ a diagram on ${\bf R}^2$: besides the
usual crossings arising naturally as projections of classical
crossings, we get virtual crossings, which arise as artefacts of the
projection: two strands lie in different places on $S_g$ but they
intersect on the plane because they are forced to do so.

Having a (virtual) knot diagram, we can smooth all classical
crossings of it in the following two ways:
$A:\mbox{\skcrossr}\to\mbox{\skcrh}$ and $B:\mbox{\skcrossr}\to
\mbox{\skcrv}$.

Thus, for a diagram $L$ with $n$ classical crossings we have $2^{n}$
{\em states}. Every state is  a way of smoothing  all (classical)
crossings. Enumerate all classical crossings by $1,\dots, n$. Then
the states can be regarded as vertices of the discrete cube
$\{0,1\}^{n}$, where $0$ and $1$ correspond to the $A$-smoothing and
the $B$-smoothing, respectively. In each state we have a collection
of circles representing an unlink. We call this cube {\em the state
cube} of the diagram $L$.

Then any for any state $s$ we have its height $\beta(s)$ being the
number of crossings resolved positively, $\alpha(s)=n-\beta(s)$
being the number of crossings resolved negatively, and the number
$\gamma$ of closed circles.

Then the Kauffman bracket is defined as

\begin{equation}
\sum_{s}a^{\alpha(s)-\beta(s)}(-a^{2}-a^{-2})^{\gamma(s)-1}
\label{kabrinit}
\end{equation}

This bracket is invariant under all Reidemeister moves except for
the first one.

The normalisation $X(K)=(-a)^{-3w(K)}\langle K\rangle$, where $w$ is
the writhe number, leads to the definition of the Jones polynomial.

The Kauffman bracket satisfies the usual relation

\begin{equation}
\langle \skcrossr \rangle=a\langle \skcrv\rangle+a^{-1}\langle
\skcrv\rangle \label{kabrrel}
\end{equation}

After a little variable change and renormalisation, the Kauffman
bracket can be rewritten in the following form:
\begin{equation}
\langle \skcrossr \rangle=\langle \skcrv\rangle-q\langle
\skcrv\rangle \label{kabr}
\end{equation}

\newcommand{\dg}{\mbox{deg}}

Here we consider bigraded complexes ${\cal C}^{ij}$ with {\em
height} (homological grading) $i$ and {\em quantum grading} $j$; the
differential preserves the quantum grading and increases the height
by $1$. The height and grading shift operations are defined as
$({\cal C}[k]\{l\})^{ij}={\cal C}[i-k]\{j-l\}$.\label{shifts}

This form is used as the starting point for the Khovanov homology.
Namely, we regard the factors $(q+q^{-1})$ as graded dimensions of
the module $V=\{1,X\}, \dg\; 1=1, \dg\; X=-1$ over some ring $R$,
and the height $\beta(s)$ plays the role of homological dimension.
Then, if we define the chain space $[[K]]_{k}$ of homological
dimension $k$ to be the direct sum over all vertices of $\beta=k$ of
$V^{\gamma(s)}\{k\}$ (here $\{\cdot\}$ is the quantum grading
shift), then the alternating sum of graded dimensions of
$[[K]]_{k}$, is precisely equal to the (modified) Kauffman bracket.

{\bf Thus, if we define a differential on $[[K]]$ preserving the
grading and increasing the homological dimension by $1$, the Euler
characteristic of that space would be precisely the Kauffman
bracket.}

\begin{re}
Later on, we shall not care about the normalisation of the complexes
by degree and height shifts to make their homology invariant under
the Reidemeister moves. It is done exactly as in \cite{Kh}.
\end{re}

We have defined the {\em state cube} consisting of circles and
carrying no information how these circles interact. Turning to
Khovanov homology, we shall deal with the same cube remembering the
information about the circle bifurcation. Later on, we refer to it
as a {\em bifurcation cube}.

The chain spaces of the complex are well defined. However, the
problem of finding a differential $\partial$ in the general case of
virtual knots, is not very easy. To define the differential, we have
to pay attention to different isomorphism classes of the chain space
identified by using some local bases.

The differential acts on the chain space as follows: it takes a
chain corresponding to a certain vertex of the bifurcation cube to
some chains corresponding to all adjacent vertices with greater
homological degree. That is, the differential is a sum of {\em
partial differentials}, each partial differential acts along an edge
of the cube. Every partial differential corresponds to some
direction and is associated with some classical crossing of the
diagram.

With each circle, we associate the tensor power of the space $V$ of
graded dimension $q+q^{-1}$, however, with no prefixed basis. With a
collection of circles, we shall associate the {\em exterior power}
of this space, as follows. With each state $s$ of height $b$, we
associate a basis consisting of $2^{\gamma(s)}$ chains. Now, we
order the circles in the state $s$ arbitrarily, fix an arbitrary
orientation on them and associate with each such circle either $1$
or $X$. With any such choice, consisting of a state, an ordering of
oriented circles and a set of elements $1$ and $X$, we associate a
chain of the complex. We can also associate elements $\pm 1$ or $\pm
X$ with any circle, which also defines a chain of our complex; this
chain differs from the corresponding chain with $1$ and $X$ by a
corresponding sign. Furthermore, we identify the chains according to
the following rule: the orientation change for one circle leads to a
sign change of a chain if this circle is marked by $\pm X$ and does
not change sign if the circle is marked by $\pm 1$; the permutation
of circles multiplies the chain by the sign of corresponding
permutation. This would correspond to taking exterior product of
vector spaces (graded modules) instead of their symmetric product.

Then for a state with $l$ circles, we get a vector space (module) of
dimension $2^{l}$. All these chains have homological dimension $b$.
We set the grading of these chains equals  $b$ plus the number of
circles marked by $1$ minus the number of circles marked by $X$.

Let us now define the partial differentials of our complex. First,
we think of each classical crossing so that its edges are oriented
upwards, as in Fig. \ref{figj}, upper right picture.

Choose a certain state of a virtual link diagram $L\subset {\cal
M}$. Choose a classical crossing $U$ of $L$. We say that in a state
$s$ a state circle $\gamma$ is incident to a classical crossing $X$
if at least one of the two local parts of smoothed crossing $X$
belongs to $\gamma$. Consider all circles
 $\gamma$ incident to  $U$. Fix some orientation of these circles according to
the orientation of the edge emanating in the upward-right direction
and opposite to the orientation of the edge coming from the bottom
left, see Fig. \ref{figj}. Such an orientation is well defined
except for the case when one edge corresponding to a vertex of the
cube, takes one circle to one circle. In such situation, we shall
not define the local basis  $\{1,X\}$; we set the partial
differential corresponding to the edge, to be zero.

\begin{figure}
\centering\includegraphics[width=300pt]{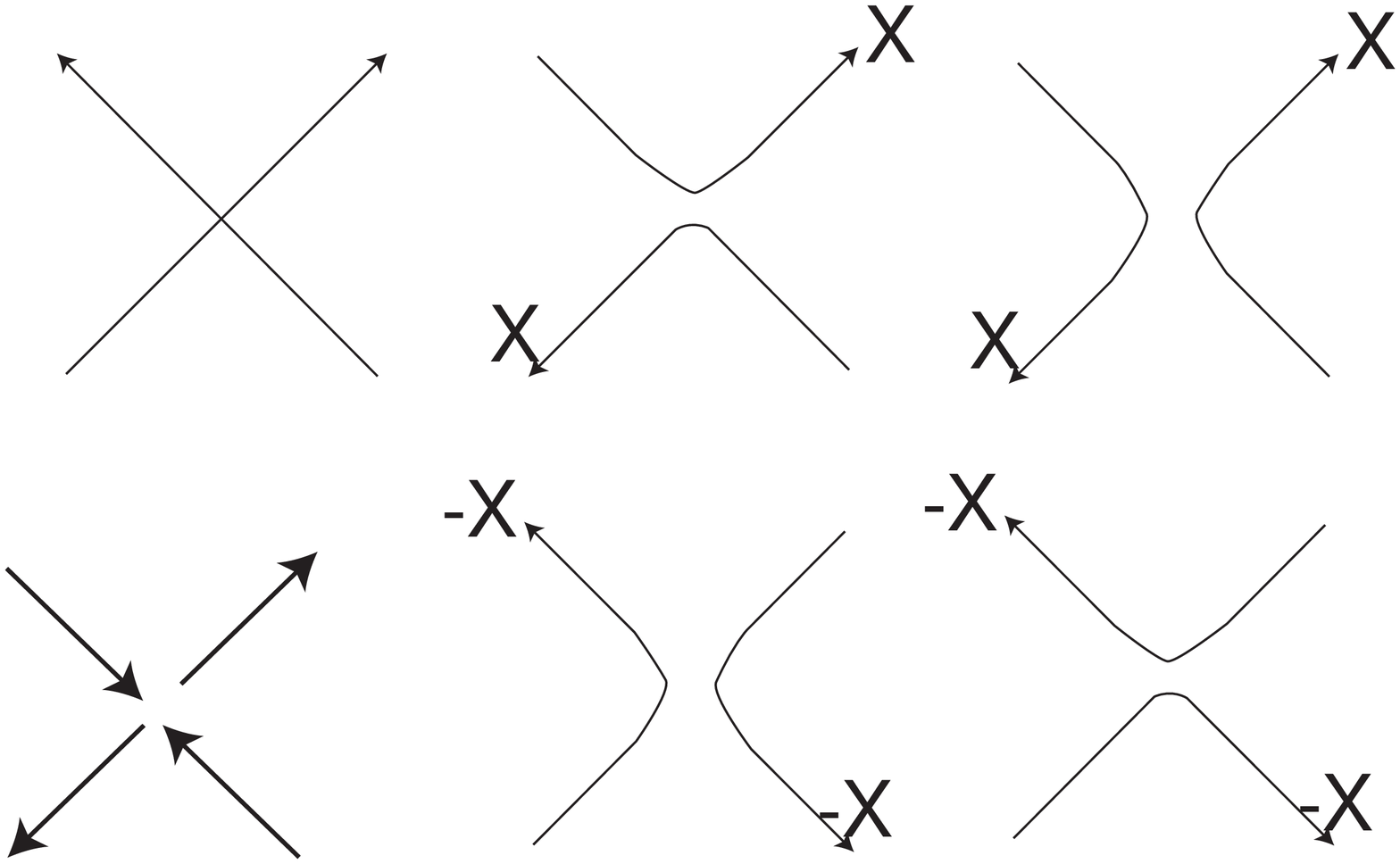} \caption{Setting
the local basis for a crossing} \label{figj}
\end{figure}


In the other situations, the edge of the cube corresponding to the
partial differential either increases or decreases the number of
circles. This means that at the corresponding crossing the local
bifurcation either takes two circles into one or takes one circle
into two. If we deal with two circles incident to a crossing from
opposite signs, we order them in such a way that the upper (resp.,
left) one is the first one; the lower (resp., right) one is the
second; here the notions ``left, right, upper, lower'' are chosen
according to the rule for identifying the crossing neighbourhood
with Fig. \ref{figj}. Furthermore, for defining the partial
differentials of types $m$ and $\Delta$ (which correspond to
decreasing/increasing the number of circles by one) we assume that
the circles we deal with are in the very initial poisitions in our
ordered tensor product; this can always be achieved by a preliminary
permutation, which, possibly leads to a sign change. Now, let us
define the partial differential locally according to the prescribed
choice of generators at crossings and the prescribed ordering.

Now, we describe the partial differentials $\partial'$ from
\cite{Izv} without new gradings. If we set $\Delta (1)=1_{1}\wedge
X_{2}+X_{1}\wedge 1_{2}; \Delta(X)=X_{1}\wedge X_{2}$ и
$m(1_{1}\wedge 1_{2})=1;m(X_{1}\wedge 1_{2})=m(1_{1}\wedge X_{2})=X;
m(X_{1}\wedge X_{2})=0$, define the partial differential $\partial'$
according to the rule $\partial'(\alpha\wedge \beta)=m(\alpha)\wedge
\beta$ (in the case we deal with a $2\to 1$-buifurcation, where
$\alpha$ denotes the first two circles $\alpha$) or
$\partial'(\alpha\wedge \beta)=\Delta(\alpha)\wedge \beta$ (when one
circle marked by $\alpha$ bifurcates to two ones); here by $\beta$
we mean an ordered set of oriented circles, not incident to the
given crossings; the marks on these circles $\pm 1$ and $\pm X$ are
given.


\begin{thm}\cite{Izv}
$[[K]]$ is a well-defined complex with respect to $\partial$; after
a small grading shift and a height shift, the homology is invariant
under generalised Reidemeister moves.
\end{thm}

Later, when we have new gradings, the differential will be defined
just by projecting this differential to the grading-preserving
subspace, namely, $\tilde
\partial'\alpha=\mbox{pr}_{\mbox{deg}=\mbox{deg}\alpha}\partial'\alpha$,
where $\mbox{pr}_{\mbox{deg}=\mbox{deg}\alpha}$ is the projection to
the subspace having all additional gradings the same as $\alpha$.
After all, we shall define $\partial$ as the sum of partial
differentials $\tilde\partial'$. We will get a set of graded groups
$Kh'_{H}$ with differential $\partial$. This differential increases
the height (homological grading), preserves the grading, and does
not change the additional gradings.

\begin{re}
The homology theory described above is initially constructed out of
planar diagrams; thus, it represents a homology theory for links in
thickened surfaces {\em modulo stabilisation}; that is, this
homology theory ``does not feel'' removable handles. However, when
we impose new gradings, we will have to fix the thickened surface,
since we will deal with its homology groups. The new complex to be
constructed for such thickened surfaces, frankly speaking, would not
be a virtual link invariant. It would rather be an obstruction for
links in thickened surfaces to decrease the underlying genus of the
corresponding surface.
\end{re}

\subsection{Usual Khovanov homology}

For the case of classical knot theory (and also some parts of
virtual knot theory) the above setup is actually not needed for
constructing Khovanov homology. One can get the chain spaces
generated by tensor powers of $V$ with appropriate grading and
degree shifts, with no care about signs as it was done in the
original Khovanov paper \cite{Kh}. Namely, one takes just the
symmetric tensor power $V^{\otimes k}$ for a vertex of a cube with
$k$ circles in the corresponding state. One also need not care about
signs: the type-$X$ generators are chosen once forever. Then it
allows to construct partial differentials just by using some
concrete formulae for $\Delta$ and $\mu$. The main difficulty we had
to overcome was the case of $1\to 1$-type partial differentials. If
no such $1\to 1$-bifurcations occur then the original construction
works straightforwardly. Namely, after splicing some minus signs,
these formulae lead to a well defined complex whose homology is the
usual Khovanov homology.

The main goal of the present paper would be to find such additional
gradings.

\section{Bourgoin's twisted knots. \\ Additional gradings}

Assume for some category (knots, virtual knots, braids, tangles) we
have a well-defined Kauffman bracket. That is, we have a set of
(classical) crossings, which can be smoothed so that the formula
\eqref{kabr} can be applied.

Consider the following generalisation of virtual knots (proposed by
Mario Bourgoin, see \cite{Bou}).

We consider knots in oriented thickenings of $2$-surfaces, the
latter not necessarily orientable. Namely, we take a $2$-surface $M$
and fix the $I$-bundle ${\cal M}$ over $M$ which is oriented as a
total fibration space, and keep both the orientation and the
$I$-bundle structure fixed.\label{fibredfixed}

We consider knots and links in such surfaces up to
stabilisation/destabilisation and refer to them as {\em twisted
links}. Virtual links constitute a proper part of twisted links
\cite{Bou}.

Note that this theory encloses as a partial case the theory of knots
in ${\bf R}P^{3}$, since ${\bf R}P^{3}\backslash\{*\}$ is nothing
but the oriented thickening of ${\bf R}P^{2}$.

Any link in ${\cal M}$ has a projection to the base space, the
latter being a four-valent graph.

Since the surface is orientable (and even oriented), there is a
canonical way for defining the $A$-smoothing and the $B$-smothing
with respect to the orientation. Thus, the formula \eqref{kabr}
gives a well-defined Kauffman bracket for such objects, which turns
out to be invariant; the proof is standard, see, e.g. \cite{Kho}.

Moreover, {\bf the approach described in the previous section gives
a well-defined Khovanov homology theory}. To this end, we have to
establish the chain space and the differentials.

Fix a cell decomposition of $M$ with exactly one $2$-cell $C$ and
choose a canonical ``upward'' direction for $C$. Then we can treat
every crossing as a classical one, that is, identify its
neighbourhood with the local picture shown in Fig. \ref{figj}.

This allows to define $[[K]]$ literally as above, and we get the
following

\begin{thm}
For twisted knots the complex $[[K]]$ is a well-defined complex with
respect to $\partial'$; after a small grading shift and a height
shift, the homology is invariant under isotopy (the orientation of
the ambient space remains fixed together with the $I$-bundle
structure); the differential $\partial'$ increases the homological
grading by $1$ and preserves the quantum grading and the additional
gradings.
\end{thm}

As shown in \cite{Izv}, the homology of this complex does not depend
on the choice of $C$ and the upward orientation.

We should mention, that there have been a lot of generalisations of
the Kauffman bracket, see e.g. Kauffman-Dye \cite{KaV4}, Manturov
\cite{MaXiJKTR}, Miyazawa \cite{Miy}.

Each of these generalisations introduces something new to the
formula for the Kauffman bracket of either topological or
combinatorial nature.

Bourgoin proposed the following generalization of the Kauffman
bracket for such surfaces.

\begin{equation}
\sum_{s}a^{\alpha(s)-\beta(s)}M^{\gamma''(s)}(-a^{2}-a^{-2})^{\gamma'(s)}
\label{Boubr}
\end{equation}
where $\gamma'$ and $\gamma''$ correspond to the number of
orienting/non-orienting circles in the state $s$, respectively.

The goal of the present section is to describe how to categorify
this invariant and then see {\em which further examples will fit
into the construction}.

In the Khovanov setup, we had $(q+q^{-1})$ instead of
$(-a^{2}-a^{-2})$. What should we have instead of $M$?

What should be the vector space categorifying this variable. As can
be seen from Khovanov's algebraic reasonings, see \cite{Frobenius},
the space corresponding to one circle should be two-dimensional.

To preserve the similarity with the initial picture, it is
convenient to  make one generator $(1)$ of this space having quantum
grading
 $q$ and the other one (which might be $X$ or $-X$) having
quantum grading to $q^{-1}$.

This is the point where {\em new} gradings come into play: with
every non-orienting circle in a Kauffman state, we associate the
space of graded dimension $q g^{-1}+q^{-1} g$, where $g$ corresponds
to the new grading. At the uncategorified level, this just means
$M=q g^{-1}+q^{-1} g$, and thus we lose no information.


At the categorified level, this means that we introduce a new
grading for the spaces: for every non-orienting loop we associate a
$Z$-grading equal to $1$ if this loop is marked by $X$ and $-1$ if
this loop is marked by $1$. For orienting loops, we have no new
gradings.

Now, let us define the new grading ($g$-grading) for the complex
$[[K]]$ as the sum of all new gradings over all non-orienting
circles.

Denote the obtained complex by $[[K]]_{g}$; this is actually nothing
but $[[K]]$ with new grading imposed. \label{Kg}

\newcommand{\dd1}{\mbox{$\dot{1}$}}
\newcommand{\ddX}{\mbox{$\dot X$}}

{\bf Notation}. Further on, we shall mark all labels belonging to
non-orienting circles by a point, that is, we write $\dd1$ and
$\ddX$ for labels $1$ and $X$ on non-orienting circles.

Here we give an example how one smoothing with dots gets
reconstructed into another smoothing; we put dots over some circles
which correspond to ``non-orienting'' curves, see Fig. \ref{dotted}.

\begin{figure}
\centering\includegraphics[width=300pt]{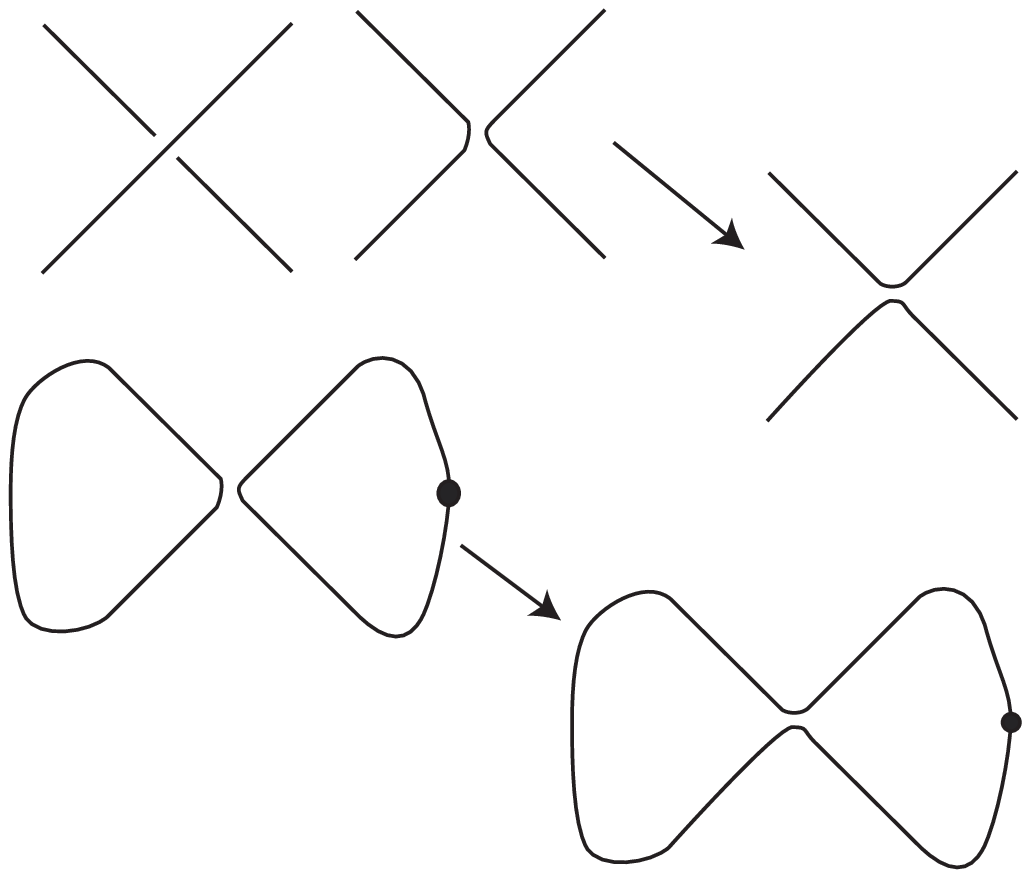} \label{dotted}
\end{figure}

Let us look how the differentials in $[[K]]$ behave with respect to
the new grading $g$. It is easy to see that

\begin{lm}
The differential $\partial$ can be uniquely represented as
$\partial'+\partial''$, where $\partial'$ preserves the new grading,
and $\partial''$ increases the new grading by $2$. \label{l1}
\end{lm}

Indeed, one can check all $m$-type and $\Delta$-type differentials,
and see that $\dd1\wedge \dd1 \to 1$, $\dd1\wedge X\to \ddX$, $X\to
\ddX\wedge \ddX$ are all increasing the grading by $2$, whence the
differential $\dd1\to X\wedge \dd1+1\wedge \ddX$ splits into two
parts, where the first one preserves the new grading, and the second
one increases that by $2$.

From Lemma \ref{l1} we get

\begin{lm}
$[[K]]_{g}$ is a well defined triply graded complex with respect to
the differential $\partial'$.
\end{lm}

\begin{proof}
Indeed, $({\partial'}^{2})$ is just the projection of
$\partial^{2}=0$ to the grading-preserving subspace.
\end{proof}

Luckily, it turns out that the homology of $[[K]]_{g}$ is invariant
after the same grading and degree shift (for old gradings) as in the
usual  case of classical knots (\cite{Kh}) or virtual knots with
oriented atoms \cite{Kho}. We shall show this more generally in the
next section.

\section{Additional grading: the general case}

The goal of our section is the following. Assume we have a space of
knots (braids, tangles, etc.) with a well-defined Kauffman bracket
and Khovanov homology. We wish to mark some circles in Kauffman's
states by dots (analogously to non-orienting cirlces in Bourgoin's
case) thus defining the new ``dotted gradings'': the dotted grading
for the state is defined as the number of all $\ddX$ minus the
number of all $\dd1$. Then we split the usual Khovanov differential
$\partial$ into two parts: the one $\partial'$ preserving the dotted
grading and the one $\partial''$ changing the dotted grading.

What are the properties this dotting should satisfy if we want the
grading to satisfy the following:

\begin{enumerate}

\item The complex $[[K]]_{g}$ is well defined;

\item Its homology (after some height and degree shift) is invariant
under isotopy (combinatorial equivalence, Reidemeister moves).

\end{enumerate}

The answer to the first question is easy: we just need that
$\partial''$ either always increase the new grading or always
decrease the new grading. Then it will guarantee $\partial'^{2}=0$.

But if we want the dots on circles to behave just as in the case of
Bourgoin so that the rules for multiplication and comultiplication
(with respect to the new grading) are:

$$m(1\wedge 1)=1;m(1\wedge X)=X;m(X\wedge 1)=X;m(X\wedge X)=0$$

$$m(\dd1\wedge 1)=\dd1;m(\dd1\wedge X)=0;m(\ddX\wedge 1)=\ddX;m(\ddX\wedge X)=0$$

$$m(1\wedge \dd1)=\dd1;m(1\wedge \ddX)=\ddX;m(X\wedge \dd1)=0;m(X\wedge \ddX)=0$$

$$m(\dd1\wedge \dd1)=0;m(\dd1\wedge \ddX)=X;m(\ddX\wedge \dd1)=X;m(\ddX\wedge \ddX)=0$$

and

$$\Delta(1)=1\wedge X+X\wedge 1$$

or

$$\Delta(1)=\dd1\wedge \ddX + dX\wedge 1$$

(depending on whether the output circles are dotted)

$$\Delta(X)=X\wedge X$$

or (when both output circles are dotted)

$$\Delta(X)=0.$$

$$\Delta(\dd1)=\dd1\wedge X$$

or

$$\Delta(\dd1)=\ddX\wedge \dd1$$

(depending on which of the two output circles is dotted)

$$\Delta(\ddX)=X\wedge \ddX$$

or

$$\Delta(X)=\ddX\wedge X$$

(depending on which of the two output circles is dotted).

The operators $m$ and $\Delta$ above are just as before (in the
categorification of Bourgoin's invariant), however, with the reasons
for putting dots completely forgotten.

Nevertheless, to have precisely this dotting, we need that {\em the
dotting of circles is additive modulo ${\bf Z}_{2}$}, that is, if we
have a $2\to 1$ bifurcation, then the number of dots for the two
circles is congruent modulo $2$ to the number of dots for the one
circle (analogously for $1\to 2$-bifurcations). We also require that
this dotting is preserved under $1\to 1$-bifurcations, that is, if a
surgery transforms one circle to one circle then this circle should
necessarily be unorienting both before and after the surgery.

The conditions above is enough for the complex $[[K]]_{g}$ to be
well defined.

Now, in order to have the invariance under the Rei\-de\-mei\-ster
moves, we have to restore the proof picture of Khovanov (or of
\cite{Izv}).

The invariance under the first Reidemeister move is based on the
following two which should held when adding a small curl:

\begin{enumerate}

\item the mapping $\Delta$ is injective

\item the mapping $m$ is surjective.

\end{enumerate}

In fact, the last two conditions hold when the small circle is not
dotted.

Indeed, consider the complex

\begin{equation} [[\skkinkr]]=\left([[\skroh]]\stackrel{m}{\to} [[\skrov]] \{1\}\right).\end{equation}

The usual argument goes as follows: the complex in the right hand
side contains a $\Delta$-type partial differential, which is
injective. Thus, the complex $[\skrov]]$ is killed, and what remains
from $[\skroh]]$ is precisely (after a suitable normalisation) the
homology of $[[\skkinkl]]$.

But $\Delta$ is injective because for any $l\in {1,X}$ we have
$\Delta(l)=l\wedge X+\langle some mess\rangle$, where the second
term $X$ in $l\wedge X$ corresponds to the small circle.

But in our situation with dotted circles, this happens only if {\em
the small circle is not dotted}. But if the small circle is dotted,
it would lead, say, to $\Delta:X\to 0$, because $\ddX\wedge \ddX$
has another dotted grading (greater by $2$ than the grading of $X$).
But when the small circle is dotted, the proof is the same.

An analogous situation happens with
\begin{equation} [[\skkinkl]]=\left([[\skrov]]\stackrel{m}{\to} [[\skroh]] \{1\}\right).\end{equation}

Here we need that  the mapping $m$ be surjective; actually, it would
suffice that the multiplication by $1$ on the small circle is the
identity. But this happens if and only if the small circle is not
dotted, that is, we have $1$, not $\dd1$.

Quite similar things happen for the second and for the third
Reidemeister moves. The necessary conditions can be summarised as
follows:


{\em The small circles which appear for the second and the third
Reidemeister move should not be dotted}.

The explanation comes a bit later.  Now, we see that this condition
is obviously satisfied when the dotting comes from a cohomology
class, and not necessarily the Stiefel-Whitney cohomology class for
non-orientable surface. Any homology class should do.

Thus (modulo some explanations given below) we have proved the
following

\begin{thm}
Let ${\cal M}\to M$ be a fibration with $I$-fibre so that ${\cal M}$
is orientable and $M$ is a $2$-surface. Let $h$ be a
$Z_{2}$-cohomology class and let $g$ be the corresponding dotting.
Consider the corresponding grading on $[[K]]$. Then for a link
$K\subset {\cal M}$ the homology of $[[K]]_{g}$ is invariant under
isotopy of $K$ in $M$ (with both the orientation of $M$ and the
$I$-bundle structure fixed) up to some shifts of the usual (quantum)
grading and height (homological grading).
\end{thm}

\subsection{Explanation for the second and the third moves}

We have the following picture for the Reidemeister move for
$[[\skrtwhh]]$:

\newcommand{\lb}{\mbox{[[}}
\newcommand{\rb}{\mbox{]]}}

\begin{equation}\begin{array}{ccc}
\lb\skrtwhh\rb\{1\} & \stackrel{m}{\longrightarrow} & \lb\skrtwvh\rb\{2\}\\
\Delta\uparrow &  & \uparrow \\
\lb\skrtwhv\rb & \longrightarrow & \lb\skrtwvv\rb\{1\}
\end{array}.
\end{equation}

Here we use the notation ${\{\cdot \}}$ for the degree shifts, see
page \ref{shifts}.

\begin{equation}\begin{array}{ccc}
\lb\skrtwhh\rb\{1\} & \stackrel{m}{\longrightarrow} & \lb\skrtwvh\rb\{2\}\\
\Delta\uparrow &  &\uparrow \\
\lb\skrtwhv\rb & \longrightarrow & \lb\skrtwvv\rb\{1\}
\end{array}.
\end{equation}
This complex contains
the subcomlex ${\cal C}'$:

\begin{equation}{\large {\cal C}'}=\begin{array}{ccc}
\lb\skrtwhh\rb_{1}\{1\} & \stackrel{m}{\longrightarrow} &\lb\skrtwvh\rb\{2\}\\
 \uparrow & & \uparrow \\
0 & \longrightarrow & 0
\end{array}\end{equation}
if {\em the small circle is not dotted}.

Here and further $1$ denotes the mark on the small circle.

Then the acyclicity of ${\cal C}'$ is evident.

Factoring ${\cal C}$ by ${\cal C'}$, we get:

\begin{equation}\begin{array}{ccc} \lb\skrtwhh\rb\{1\}\slash_{1=0} &
\longrightarrow&  0\\
\Delta\uparrow & & \uparrow \\
\lb\skrtwhv\rb & \longrightarrow &
\lb\skrtwvv\rb\{1\}\end{array}.\end{equation}

In the last complex, the mapping $\Delta$ directed upwards, is an
isomorphism (when our small circle is not dotted). Thus the initial
complex has the same homology group as $[[\skrtwvv]]$. This proves
the invariance under $\Omega_{2}$.

The argument for $\Omega_{3}$ is standard as well; it relies on the
invariance under $\Omega_{2}$ and thus we should also require that
the small circle is not dotted.

\section{More gradings; more examples}

We have listed the necessary conditions for the dotting to give such
a grading that $[[K]]_{g}$ is invariant (up to some shifts); the
conditions are quite natural: additivity of dots modulo ${\bf
Z}_{2}$ and triviality of small circles for all types of
Reidemeister moves. We have actually missed one condition we assumed
without saying. Namely, {\em in the pictures corresponding to the
Reidemeister moves, the similar arcs are dotted similarly}.

This means, for example, that for the second Reidemeister move the
smoothing $\skrtwvv$ gives two branches which should have the same
dotting as the two branches of $\skrtwvv$. The same follows for all
the three moves.

Thus, we introduce the {\em dotting axiomatics}. Namely, assume we
have some class of objects with Reidemeister moves, Kauffman bracket
and the Khovanov homology (in the usual setup or in the setup of
\cite{Izv}). Assume its circles can be dotted in such a way that the
following conditions hold:

\begin{enumerate}

\item The dotting of circles is additive with respect to $2\to
1$-bifurcations, and it is preserved under $1\to 1$-bifurcations.

\item Similar curves for similar smoothings of the RHS and the LHS
of any Reidemeister move have the same dotting

and

\item Small circles appearing for the first, the second, and the
third Reidemeister moves are not dotted.

\end{enumerate}

Let us call the conditions above {\em the dotting conditions}.

\begin{thm}
Assume there is a theory with Khovanov complex $([[K]],\partial)$
such that the Kauffman states can be dotted so that the dotting
conditions hold. Define $[[K]]_{g}$ as before (see page
\pageref{Kg}).

1) Then the homology of $[[K]]_{g}$ is invariant (up to a degree
shift and a height shift).

2) For any operator $\lambda$ on the ground ring, the complex
$[[K]]_{g}$ is well defined with respect to the differential
$\partial'+\lambda\partial''$, and the corresponding homology is
invariant (up to well-known shifts).

3) Moreover, if we have several dottings $g_{1}, g_{2}, \dots,
g_{k}$ so that for each of them the dotting condition holds, then
the complex $K_{g_{1},\dots,g_{k}}$ with differential
$\partial_{g_{1},\dots, g_{k}}$ defined to be the projection of
$\partial$ to the subspace preserving all the gradings, is
invariant.

\label{mainthm}
\end{thm}

\begin{proof}
The first part of the theorem follows from the reasonings above.

Now, for the differential
$\tilde\partial=\partial'+\lambda\partial''$ we have
$(\tilde\partial)^{2}={\partial'}^{2}+\lambda(\partial'\partial''+\partial''\partial')+\lambda^{2}{\partial''}^{2}$;
the expression in the right hand side gives the projections of
$(\partial)^{2}=(\partial'+\partial'')^{2}$ to three subspaces of
corresponding gradings taken with some coefficients (here
$1,\lambda,\lambda^{2}$). Since $(\partial)^{2}=0$, all projections
are zeroes. The invariance of the homology is proved as above. The
main thing is that the mapping $m$ is surjective and $\Delta$ is
injective.

The proof of the last statement is analogous to the proof with only
one grading. Again, it is enough to mention that $m$ remains
surjective and $\Delta$ remains injective.

\end{proof}

\subsection{Examples}

One example (already published in the note \cite{Doklady}) deals
with the following situation. Consider a fixed thickened surface
${\cal M}$ which is the total space of an $I$-fibre bundle over some
$2$-manifold $M$, not necessarily orientable. We assume the
orientation of ${\cal M}$ and the $I$-bundle structure fixed.

Consider all ${\bf Z}_{2}$-cohomology classes $H^{1}({\cal M})$
(there are finitely many of them). For knots in ${\cal M}$, each of
these classes generates a dotting for circles (see page \ref{Kg}) in
the Kauffman states, thus, it defines gradings for $[[K]]$. Call
these gradings {\em additional} (with respect to the two usual
Khovanov gradings). Denote the obtained complex by $[[K]]_{gg}$ and
the projection of the differential $\partial$ by $\partial_{gg}$.

\begin{thm}
The homology of $[[K]]_{gg}$ with respect to $\partial_{gg}$ is an
invariant of $K$.
\end{thm}

Consider the category $T$ of (classical or virtual) tangles with
$2k$ open ends. Then the construction above allows to make the
following dotting on the Kauffman homology.

Fix some number $l$ and mark some of the tangle ends by some of $l$
colours $1,2\dots,l$.

Couple the endpoints of the tangle in an arbitrary way (so that any
tangle closes into a classical or virtual knot).

Having done this, for any tangle $t\in T$, we can consider its
closure $Cl(t)$. It acquires a dotting from $l$ colours, thus we get
$l$ additional gradings for the Khovanov complex; denote the
obtained complex by $[[Cl(t)]]_{dd}$, and denote the corresponding
differential by $\partial_{dd}$.

From the above, we see that

\begin{thm}
For any fixed endpoint coupling, the homology of $[[Cl(t)]]_{dd}$ is
an invariant of $t$.
\end{thm}

A particular case of this refers to long classical (and virtual)
knots.

Namely, if we deal with {\em long} virtual knots, this grading will
lead to a new invariants. Note that {\em long} virtual knots do not
coincide with {\em compact} virtual knots, see e.g., \cite{Malng}.
There are non-trivial long virtual knots having only trivial
classical closures. Say, it is easy to construct two classical
$2-2$-tangle with the same classical closures and different virtual
closures.

As for classical knots, thinking of them from the ``long'' point of
view seems to be very prospective. In our case, if we take long
classical knots and put one dot on one end, thus defining a new
grading. This will split the usual Khovanov differential $\partial$
into $\partial'+\partial''$. The only circle which can support the
new grading is the one obtained by closing the only long arc. It
exists in every state, and it can be marked either by $\ddX$ or by
$\dd1$. If we just take $\partial'$, then it would split the initial
Khovanov complex into two parts: the one with ${\ddX}$ and the one
with ${\dd1}$ with no differential acting from one part to
another.

This is nothing but the usual {\em reduced Khovanov homology}.

However, if we take not just $\partial'$, but
$\partial'+\lambda\partial''$ for some ring $R$ where $\lambda$ is a
zero divisor (say, 2 in the ring ${\bf Z}_4$).

This defines new invariants of ordinary knots (or links with one
marked component).

However, it seems to be much more interesting when we pass from
usual long knots to cables. Namely, having a long classical knot
(assume it to be {\em framed}), we can take its $n$-cabling. Then
for any dotting and for any closure the new homology groups will be
invariants of the initial (long) classical knot.

One more example refers to {\em rigid virtual knots}. We consider
virtual knot diagrams up to all Reidemeister moves and all detours
preserving the Whitney index of the curve. Namely, we prohibit the
following {\em first virtual Reidemeister moves}:
$\skkinkv\to\skcurl$. Rigid virtual knots are of interest because
all quantum invariants of classical knots (which can not be
generalised for generic virtual knots) can be generalised in full
totality for rigid virtual knots.


For such knots, since the first virtual Reidemeister move is
forbidden, in any Kauffman state for any circle the number of
self-intersections modulo $2$ for such circles is invariant. It
defines well a dotting, thus giving one new grading for rigid
virtual knots (hence, for zero-homologous virtual knots as well).

\subsection{Braids}

It is a very intriguing question to get new gradings for classical
knots (without going to long knots).

We are not going to consider braids just as a partial case of
tangles and put various dots on the ends of the braid. We think of a
braid as a source of constructing knot invariants via Markov moves.

Thus, a {\em closed braid} can be viewed of as a special kind of
link in a thickened annulus $S^{1}\times I\times I$. This annulus
has non-trivial cohomology group $H^{1}(S^{1}\times I\times I,{\bf
Z}_{2})={\bf Z}_{2}$. From this we get an additional grading, thus
having a complex $[[Cl(B)]]_{g}$ with  differential $\partial'$;
here $Cl(B)$ is the closure of a braid $B$. It is obvious that the
homology of this complex is well defined not only under braid
isotopies, but also under braid conjugations, since they preserve
the closure.

Thus, in order to get a knot invariant, we have to overcome the
second Reidemeister move (adding a new loop). Unfortunately, if $K'$
is obtained from $K$ by a second Markov move then the homology of
$Cl(K')$ should not coincide with the homology of $Cl(K)$. The
reason is that the move we perform is the first Reidemeister move,
and the small circle that appears is dotted.

However, this allows to extract the difficulty for proving the
invariance of the the new dotted (grading) homology for knots in its
{\em pure form}: the only obstacle we get is the first
Rei\-de\-mei\-ster move.

Hopefully, the homology of this space with extra gradings behaves in
a predictable manner under the Markov move, maybe, after some
stabilisations.

We shall return to this question while speaking about filtrations
and spectral sequences.

\subsection{Further gradings}

The construction above takes into account only ${\bf
Z}_{2}$-homology classes (unlinke the construction of \cite{APS})
where the {\em homotopy} information of Kauffman state circles was
taken into account to construct a grading.

More homology information can be taken into account in the following
manner.

Assume we have only one non-trivial cohomology class (say, we live
on the thickened annulus or deal with long knots with one dot on one
end).

Then such an object has $H^{1}={\bf Z}$. In what follows, we were
using only the ${\bf Z}_{2}$ information for constructing our
differentials.

We shall now use the ${\bf Z}$-cohomology information to introduce
the {\em secondary gradings} as follows.

If the usual grading coming from the ${\bf Z}_{2}$-homology class is
non-trivial, then we decree the secondary grading to be zero. If the
first grading is trivial, then we look at the value of the
cohomology group not over ${\bf Z}_{2}$, but over ${\bf Z}_{4}$ and
then we set the secondary grading to be $0$ if the cohomology class
is trivial modulo ${\bf Z}_{4}$ and $1$ if it is equal to $2$ modulo
${\bf Z}_{4}$. Analogously, in the case when the primary and the
secondary gradings are both zero, we define the ternary grading to
be $1$ or $0$ depending on the value of the ${\bf Z}_{8}$-cohomology
(of course, if one of them was not zero, we set all further gradings
to be zero).


This defines a family of further gradings on circles which answers
the question what is the maximal power of $2$, the corresponding
value of the co\-ho\-mo\-lo\-gy is equal to. For instance, such
gradings can be all zeroes (say, if the circle is trivial) or
$(1,0,0,\dots)$ or $(0,1,0,0,\dots)$ or $(0,0,1,0,0,\dots)$, etc.

These gradings define corresponding dottings and gradings for all
elements $1$ and $X$ (as before, we count the gradings for $X$ with
plus, and the gradings for $1$ with minus).

This defines a multigrading on the complex (chain set) $[[K]]$.
Denote the obtained chain set by $[[K]]_{mg}$. The usual
differential $\partial$ for $[[K]]$ splits into two parts: the one
$\partial'$ preserving the new multigrading and the one $\partial''$
not preserving the grading.

\begin{lm}
For any of the new gradings, the differential $\partial''$ either
preserves it or increases it by $2$.
\end{lm}

\begin{proof}
Indeed, assume we have a bifurcation $2\to 1$ or $1\to 2$. Such a
bifurcation may behave in two ways with respect to the new gradings
on circles: either it preserves the total set of gradings (each
considered modulo ${\bf Z}_{2}$) as in the case $(1,0,\dots )\wedge
(1,0,\dots,)\to (0,0\dots)$, or it changes it, as in the case
$(1,0,\dots, )\wedge (1,0,\dots )\to (0,1,\dots)$. In the second
case the parity in one grading (in our case, the second) is
violated, thus, $\partial'$ equals zero.

In the first case we may think that our differetial behaves in the
same way with {\em all} the gradings separately, which returns us to
the case of different gradings coming from different homology
classes.
\end{proof}

The above reasonings lead us to the following
\begin{thm}
The homology of $[[K]]_{mg}$ with respect to $\partial'$ is an
invariant in the corresponding category.
\end{thm}

Analogously, one may consider the case when we have $H^{1}$ of rank
greater than one.

\section{Khovanov's Frobenius theory}

The Khovanov theory for classical knots has some natural
generalisations, some of them were first discovered by Khovanov.
Here we briefly discuss the generali\-sa\-tion of them for the case
of {\em knots in thickened surfaces and additional gradings}. The
corresponding results without additional gradings were published in
\cite{Kho, Izv}.

\newcommand{\RR}{\cal R}
\renewcommand{\AA}{\cal A}

Let $\RR,\AA$ be commutative rings, and let $\iota: \RR\to \AA$ be
an embedding, such that
 $\iota(1)=1$. The restriction functor mapping $\AA$-modules to
 $\RR$-modules has a right conjugate and a left conjugate: the induction functor
 $Ind(M)=\AA\otimes_{\RR}M$ and the coinduction functor.
 $CoInd(M)=Hom_{\RR}(\AA,M)$. One says that $\iota$ is a Frobenius embedding if these two
functors are isomorphic. Equivalently: the embedding $\iota$ is
Frobenius, if the restriction function has a two-sided dual functor.
In this case one says also that the ring $\AA$ is a {\em Frobenius
extension} of $\RR$ by means of $\iota$.

In  \cite{Frobenius}, Khovanov asked the question: to find a couple
of linear spaces $(\AA,\RR)$ such that, taking $\RR$ as the basic
coefficient ring and a Frobenius extension $\AA$ over  $\RR$ as the
homology ring of the unknot, we would be able to construct a link
homology theory ``in the same way'' as the usual homology theory.

Here ``in the same manner'' means that we consider the state cube,
where at each vertex we put a tensor power of $\AA$ (over $\RR$),
corresponding to the number of circles in the given state, and
define the partial differentials by means of $m$ and $\Delta$
(multiplication and comultiplication), and then put signs on the
edges of the cube and normalise the whole construction by height and
grading shifts (he did not use wedge product or involution in the
Frobenius algebra).

Khovanov showed that the invariance under the first Reidemeister
move requires that $\AA$ is a two-dimensional module over $\RR$ and
gave necessary an sufficeint conditions for the existence of such an
invariant link homology theory.

{\em Note that in the present section we shall mainly work with the
classical notation of Khovanov, that is, we use symmetric tensor
powers and then add minus signs to the cube}, thus restricting
ourselves for the case when no $1\to 1$-bifurcations in the state
cube occur. We have partially generalised Khovanov Frobenius theory
for the case of arbitrary virtual knots, and we shall return to that
case in the end of the present section.

In \cite{Frobenius}, it is also shown that any link homology theory
of such sort can be obtained by means of some operations (basis
change, twisting and duality) from the following solution called
{\em universal}:

\newcommand{\grad}{\operatorname{deg}}

\begin{enumerate}

\item ${\RR}={\bf Z}[h,t]$.

\item ${\AA}={\RR}[X]\slash (X^{2}-hX-t),$

\item $\grad X=2,\grad h=2, \grad t=4$;

\item $\Delta(1)=1\otimes X+X\otimes 1 - h 1\otimes 1$

\item $\Delta(X)=X\otimes X+t 1\otimes 1$.

\end{enumerate}

As we see, the multiplication in the algebra $\AA$ preserves the
grading, and the comultiplication increases this by $2$.

We omit the normalisation regulating the corresponding gradings.

First note that this Frobenius theory contains (as an important
partial case) the Lee-Rasmussen theory, see \cite{Lee,Ras}, when we
specify $t=h=1$. The Lee-Rasmussen theory, has one grading less:
indeed, the differentials here do not respect the quantum grading.

We call the theory constucted above the  {\em universal
$(\RR,\AA)$-construction}. The corresponding homology of a
(classical) link $L$ is be denoted by $Kh_{U}(L)$.

The main question we address in the following section is: {\em how
to split the differentials above into $\partial'$ and $\partial''$}?

Note that if we introduce the new grading just by dotting and then
counting the number of $\ddX$ minus the number of $\dd1$, the
differential $\partial$ (which is some tensor product (or wedge
product) of one $\Delta$ or one $\mu$ with the identity operator)
would not behave so nicely with respect to the new grading. Namely,
the mapping $\Delta$ may take $X$ to the sum $\ddX\wedge
\ddX+\dd1\wedge\dd1$, see Fig. \ref{dots2}.

\begin{figure}
\centering\includegraphics[width=200pt]{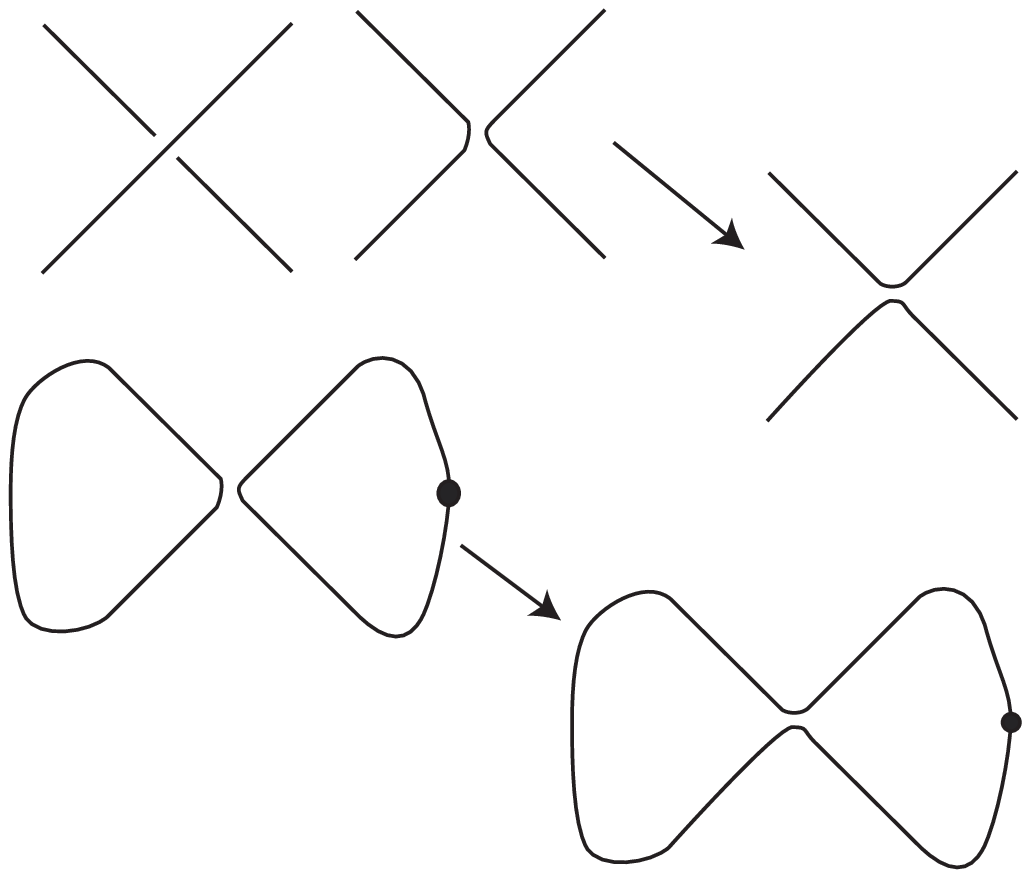}
\label{dots2}
\end{figure}

The mapping to the first term {\em increases} the grading whence the
mapping to the second term {\em decreases} it.

Thus, we have to repair the dotted grading. The correct answer is:
define the dotted grading $gr$ as the difference between
${\#}\ddX-{\#}\dd1$ plus half the total degree of monomials in $t$
and $h$.

There is a trick with $\lambda$, which goes as follows. Denote the
usual Khovanov differential by $\partial$, and denote the
``Frobenius addition'' containing $h$ and $t$ by $\partial_{F}$ so
that we totally have $\partial+\partial_{F}$. According to our
rules, if some circles are dotted, and the Khovanov (Frobenius)
theory is well established then we can introduce the new ``dotted
grading'' $gr$ as before, which splits the differential into two
parts $\partial=\partial'$.

\begin{thm}
Consider the basic ring ${\bf Z}[h,t,\lambda|\lambda h=\lambda t=0]$
Then the homology of the Khovanov Frobenius complex with respect to
the differential $\partial_{F}+\lambda\partial''$ is invariant.
\end{thm}

The proof goes as follows. We only need to  mention that is that the
square of this differential equals zero, because in the expression
$(\partial_{F}+\lambda\partial'')^{2}$ the interaction between the
``Frobenius part'' of $\partial_{F}$ and $\lambda\partial''$ gets
cancelled. This proves that the complex is well defined with respect
to the differential $\partial''$. However, one of our goals is to
approach the Lee-Rasmussen theory, which is defined over ${\bf Q}$
with $h=t=1$. For these purposes, the approach above is not
satisfactory.

Then, the terms in the differential corresponding to the ``usual''
multiplication and comultiplication (without new $t$ and $h$) behave
as before. Also, we know the behaviour of the grading when we have
no dotted circles; it is correlated by degrees of $h$ and $t$.

Consider the remaining cases.

$m:\ddX\otimes \ddX\to  t\cdot 1,m:\ddX\otimes X\to t\cdot \dd1
\Delta: \ddX\to t\cdot 1\otimes 1,\Delta:\ddX\to \dd1\wedge \dd1$.

But, looking carefully at the usual quantum grading, we shall see
that the dotted grading decreases only in the case when the usual
quantum grading increases. Namely, for $\ddX\times \ddX\to 1$ we
increase the usual grading by four (because the latter $1$ is indeed
shifted by $1$. So, the idea is to add $\frac{4}{2}=2$ to our usual
dotted grading to get a better dotted grading. Thus we get the
following

\begin{lm}
The differential $\partial$ either increases the grading $\partial$
by $2$ or does not increase it at all.
\end{lm}

But, looking carefully at the usual quantum grading, we shall see
that the dotted grading decreases only in the case when the usual
quantum grading increases. Namely, for $\ddX\times \ddX\to 1$ we
increase the usual grading by four (because the $1$ in the
right-hand side is indeed shifted by $1$. So, the idea is to add
$\frac{4}{2}=2$ to our usual dotted grading to get a better dotted
grading.

Let us look at our dotted grading more carefully. Denote the former
dotted grading by $gr'$, and let us construct the true dotted
grading $gr$ by varying $gr$.

We count the usual quantum grading. It is equal to
$tot(1)-tot(X)+h$, where $tot(1)$ is the total number of circles
marked by $1$ or by $\dd1$, $tot(X)$ is the total number of circles
marked by $X$ or by $\ddX$, and $h$ is the height. Then we set

$$gr=gr'+\frac{tot(1)-tot(X)+h}{2}=\frac{{\#}\ddX+{\#}1-{\#}{\dot 1}-{\#}X+h}{2}.$$

\begin{lm}
The differential $\partial$ defined above either preserves $gr$ or
increases it by $2$.
\end{lm}

The proof follows from a direct calculation.

Then it is possible to split $\partial$ into $\partial'$ (preserving
the grading) and $\partial''$ increasing that by $2$, and consider
the dotted homology of $[[K]]_{g}$ with respect to $\partial$. This
homology will be invariant.

If we look at this grading more carefully, we will see that the new
``Frobenius'' mappings vanish when they are applied to sets of usual
(not dotted) circles.

Namely, for $\partial: X\otimes X\to t\cdot 1$ we have: $gr'$ does
not change, whence the usual grading [coming from counting
$tot(1)-tot(X)+h$] increases.

This means, that if we have no dots at all, the differential
$\partial'$ coincides with the usual Khovanov differential (without
$h$ and $t$).

Considering the Lee-Rasmussen theory for $t=h=1$, we get a complex
$[[K]]_{LR}$ with a differential $\partial_{LR}$ which coincides
with the usual Khovanov differential in the case of classical knots.
Note that the complex $[[K]]_{LR}$ has two gradings: the height and
the grading $gr$ (the quantum grading was lost).

However, in the dotted picture, this differential has some other
interesting terms, like $\ddX\otimes \ddX \to 1$.

\subsection{Yet another definition of the Khovanov homology}

If we look at the complex constructed above from in the case we have
no additional (dotted) gradings at all, we see that the new grading
prohibits exactly those parts of the differential $\partial_{\Phi}$
which deal with $t$: e.g., $X\times X\to t\cdot 1$ does not change
the dotted grading, but it does change the usual quantum grading if
we forget about $t$.

Thus, {\em the definition above with $t=1$ leads to the usual
Khovanov homology if no circle is dotted}.

On the other hand, if many circles are dotted, this is a sort of
Lee-Rasmussen homology theory.

It is interesting that we can use a mixture to get another
definition of the Khovanov homology theory. Namely, take a knot
diagram $K$ and put dots on circles in an arbitrary way. Then for
every dotted circle change the notation: replace $\dd1$ by $\ddX$
and vice versa. The resulting complex would be precisely the
Khovanov complex up to some renormalisation in the new grading which
becomes coincident with the usual quantum grading.

This effect is interesting because it allows one to handle the
situation with braids: whenever we perform the second Markov move,
we replace $\dd1$ by $\ddX$, which leads to the injectivity of
$\Delta$ and surjectivity of $m$. Unfortunately, this gives us no
new homology theory, but it allows one to look at the usual Khovanov
homology from another point of view.

\subsection{Khovanov Frobenius theory modulo ${\bf Z}_{2}$ in the general case}

The aim of the present section is to define the differential
$\partial_{F}$ generalizing the theory described above for the case
of arbitrary virtual knots in the ${\bf Z}_{2}$ case. We shall
describe the difficulties that occur in the general case.

The main difficulty here is to define the differential corresponding
to the $1\to 1$-bifurcation.

We start up with the chain structure of the complex. First, we
assume for simplicity $h=0$, the case of generic $h$ will be
considered afterwards.

We deal with the ring $R={\bf Z}[t]$, where $t$ has grading $4$.

With every circle in every Kauffman state we associate the graded
module $V$ over $R$ freely generated by $1$ of grading $0$ and $X$
of grading $2$ ($t$ has grading $4$, as above). The generator $1$ is
assumed to be fixed for any circle; the generator $X$ depends on the
orientation of the circle as before.

With each Kauffman state with $n$ corresponding circles, we
associate the $n$-th exterior power of $V$, and we define the
following operations ``muliplication and comultiplication'' just as
before, however, corrected by terms containing $h$:

$m(1_{1}\wedge 1_{2})=1,m(X_{1}\wedge 1_{2})=m(1_{1}\wedge
X_{2})=X,$

$m(X_{1}\wedge X_{2})=0$

$\Delta(1)=1_{1}\wedge X_{2}+X_{1}\wedge 1_{2}$

$\Delta(X)=X_{1}\otimes X_{2}+t 1_{1}\otimes 1_{2},$ where it is
assumed (as before) that we deal with the first two circles in the
tensor product, and the first one is left (resp., upper), whence the
second one is left (resp., lower).

For all $1\to 1$-bifurcations, we set the differential to be equal
to zero.

For all other bifurcations ($2\to 1$ or $1\to 2$), we define the
differential $\partial$ just as in section 3.

Denote the resulting set of chain spaces for a given virtual knot
$K$ by $[[K]]_{t}$.

\begin{thm}
The differential $\partial$ defines a complex structure on
$[[K]]_{t}$, so that the homology of $[[K]]_{t}$ with respect to
$\partial$ is an invariant of the link $K$.
\end{thm}

The well-definedness proof actually repeats the main points of
\cite{Frobenius} together with those in \cite{Izv}: one should
consider all $2$-faces of the corresponding cube and prove that they
anticommute. The proof of the invariance under Reidemeister moves
follows from the surjectivity of $m$ and injectivity of $\Delta$.

However, here we do not touch on the variable $h$. The reason why
the trick proposed in \cite{Izv} behaves nicely when we add the
variable $t$ is the following: both in the usual Khovanov homology
theory and in the Frobenius theory with some $t$ and $h=0$, the
involution on the space $V=\{1,X\}$ defined by $1\mapsto 1,X\mapsto
-X$ behaves well with respect to the operations $\Delta$ and $m$: it
changes signs of $\Delta$ and preserves the sign of $\mu$.

However, when we add a new variable $t$, we will not see this effect
any more: the mapping $\Delta$ takes $1\wedge 1\to 1\wedge X+X\wedge
1-h\cdot 1\wedge 1$. Here the involution $X\to -X$ changes the sign
of one part ($1\wedge X+ X\wedge 1$) and preserves the other part
($h\cdot 1\wedge 1$).

Also, the routine check of the well-definedness (as in \cite{Izv})
of the complex, that is, anti-commutativity of the $2$-faces of the
cube, leads to an example shown below (we are citing \cite{Izv}, see
Fig. \ref{norient}) for the case $t=0$.

\begin{figure}
\centering\includegraphics[width=300pt]{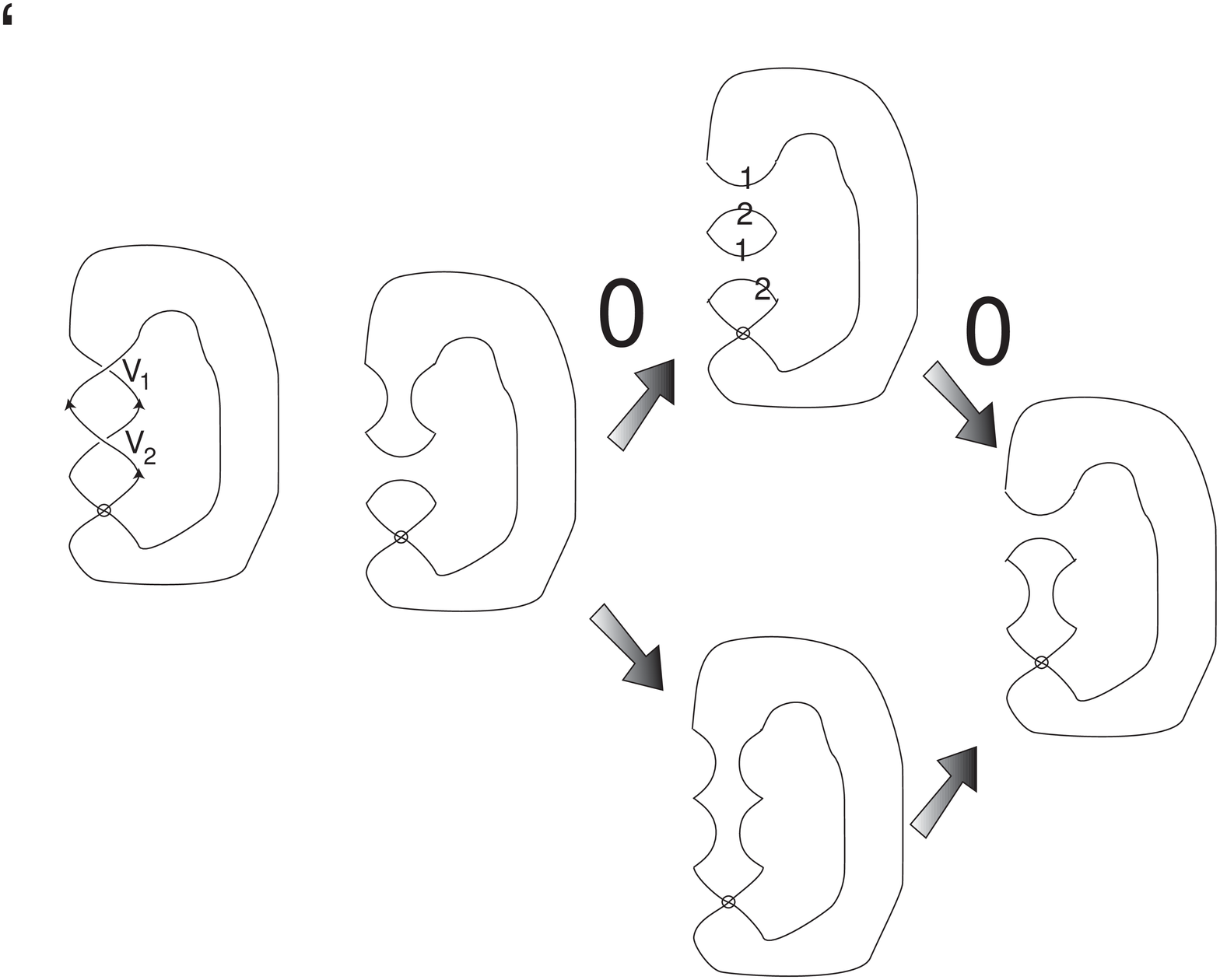} \caption{A face
of the cube} \label{norient}
\end{figure}

For the lower composition, we have the identical zero map by
definition. Substituting $X$ into the upper composition, we get $\pm
X\wedge X$ at the first step and $0$ at the second step.
Substituting $1$, we first get $1_{1}\wedge X_{2}+X_{1}\wedge 1_{2}$
here the index refers to the number of circle (the first circle is
the big one), and the second index refers to the crossing number.
While passing to the second crossing $V_{2}$ the circles change
their roles: the first circle becomes the lower one, and the second
circle becomes the upper one. Moreover, for the first circle we get
a basis change: $X$ maps to $-X$. Thus we get $-X\wedge 1+ 1\wedge
X$, which is taken to zero by the  multiplication $m$.

The example above is in fact the key example of \cite{Izv}; it works
without any changes when $h=0$ (because $t$ does not appear in the
comultiplication of $1$ or in the multiplication of $1\wedge X$).

But in the case $h=0$ it does appear, and this would lead to the
fact that the $1\to 1$-bifurcation should not be zero any more. We
will in fact need to introduce a new variable being the square root
of $h$.

On the other hand, $h$ itself should be treated in a special way so
that the multiplication $m$ and comultiplication $\Delta$ behave
nicely with respect to $1\mapsto 1, X\mapsto X$.


We shall consider this problem in a separate publication.

\subsubsection{The ${\bf Z}_{2}$-case}

We first consider the ${\bf Z}_{2}$-case solution given in
\cite{Kho}. First note that there is no difference between $\wedge$
and $\otimes$, and we shall use the notation $\otimes$.

We set all $1\to 1$ type partial differentials to be zero.

Here will show how the square root of $h$ appears. Of course, in
this case we shall not need exterior products and control the signs.
Consider the basic ring of coefficients ${\bf Z}_{2}[t,c]$ with
$\deg t=4, \deg c=1$ (we assume $c^{2}=h$). Now, consider Fig.
\ref{norient}. We have the following situation: in the lower
composition we have two maps corresponding to $1\to 1$ bifurcations,
thus the corresponding matrix should look like $I\cdot I$; in the
upper part we have the composition of two matrices $\Delta$ and then
$\mu$. Starting with $1$, we get $\Delta(1)=1\otimes X+X\otimes 1+h
1\times 1$. Multiplying, we see that $X\otimes 1$ and $1\otimes X$
cancel each other, and the only remaining term is $h\cdot 1$. Now,
if we start with $X$, we get $X\to X\otimes X+t\cdot 1\otimes 1$.
After the multiplication, we get $h X+t+t=h\cdot X$ (we are dealing
with the ${\bf Z}_{2}$ case). Now we see that the corresponding
transformation matrix looks like
\begin{equation}
\left(\begin{array}{c} 1 \\ X \end{array}\right)\mapsto
\left(\begin{array}{cc}h & 0 \\ 0 & h\end{array}\right) \cdot
\left(\begin{array}{c} 1 \\ X\end{array}\right)
\end{equation}

For this scalar matrix $h\cdot Id$ we set the bifurcation
corresponding to the $1\to 1$-mapping to be $c\cdot Id$, and then
any face of the bifurcation cube corresponding to Fig. \ref{norient}
will (anti)commute. Then it is not difficult  to see (see
\cite{Kho}) that with this {\em scalar} $1\to 1$-bifurcation matrix,
all other faces (anti)commute as well.

Now, the dotted gradings $gr$ appear straigthforwardly by counting
{\em monomials in $t$ and $c$} and correcting $gr'$ by using this
monomials. Denote the obtained homology by $Kh(K)_{tc}$.

Note that the degree of $c$ is $1$, so we will have {\em
half-integer gradings}. This immeadiately leads to the following

\begin{thm}
If $Kh(K)_{tc}$ has a non trivial homology of half-integer
additional grading then $K$ has no diagram with orientable
corresponding atom. In particular, the knot $K$ is not classical.
\end{thm}

\subsubsection{The general case}

Now we turn to the general case of the ring ${\bf Z}$, and we have
to handle the faces of the cube corresponding to Fig. \ref{norient}.

\section{Gradings or filtrations? The spectral sequence}

Since the works of Lee \cite{Lee} and Rasmussen \cite{Ras}, spectral
sequences play a significant role in knot homology. Sometimes it
turns out that studying convergence of a spectral sequence leads to
some interesting and deep invariants such as Rasmussen's invariant,
which is applicable to estimating the Seifert genus and the $4$-ball
genus of classical links.

The Lee-Rasmussen spectral sequence starts with the Khovanov
homology and ends up with some two-term homology which carries a
nice information.

Recently (see \cite{BarNatanJKTR}), it was discovered that the
spectral sequence of Lee-Rasmussen does not converge after
$E_{3}$-term, and that there are some nice torsions in Khovanov
homology which survive after the $E_{3}$-term of the spectral
sequence.

Our goal here is to construct a spectral sequence from the
``complicated'' theory with new dotted gradings to the ``simple''
(Khovanov) theory. Thus, in some sense our spectral sequence will
behave with respect to the usual Khovanov homology as Khovanov
homology itself behaves with respect to the Rasmussen homology.

It would also be very interesting to inspect two spectral sequences
converging from the ``complicated'' theory to the Rasmussen theory.

The argument of the present section is standard. In all cases
described above when we
 deal with one new (dotted) grading, the {\em
old} differential $\partial=\partial'+\partial''$ in the complex
$[[K]]_{g}$ does not decrease the new grading.

Thus, let us introduce the (dotted) {\em filtration} on the chain
spaces as follows: we set $[[K]]_{g}^{n}=\{c\in [[K]]_{g}|gr(c)\ge
n\}$. Then we have $[[K]]_{g}^{\infty}\subset \dots
[[K]]_{g}^{2}\subset [[K]]_{g}^{1}\subset [[K]]_{g}^{0}\subset
[[K]]_{g}^{-1}\subset \dots \subset [[K]]_{g}^{-\infty}$.

The usual differential $\partial$ respects this filtration. This
leads to the following

\begin{thm}
For any field of coefficients, there is a spectral sequence whose
$E_{1}$-term is isomorphic to $[[K]]$ with the first differential
$\partial$, the $E_{2}$-term isomorphic to the usual Khovanov
comology, so that this spectral sequence converges to the homology
of $[[K]]_{g}$ with respect to $\partial'$.
\end{thm}

The argument proving this theorem is standard. We also conjecture
that all terms of this spectral sequence are invariants (of knots,
braids, tangles) in the corresponding category.

It would be very interesting to know whether some terms of the
spectral sequence described above survive after the braid
stabilsations. In this case we would be able to hope to construct
gradings for usual knots without going into the long category.

Returning to the Lee-Rasmussen theory, we see that in the dotted
case, we have two complexes: the usual Khovanov complex and the
complex $([[K]]_{LR}, \partial_{LR})$ with homology $H(K)_{LR}$.
They coincide in the case when we have no dotting, but they differ
in the case when we have dotting.

Quite in the usual manner one proves
\begin{thm}
For the field ${\bf Q}$, there is a spectral sequence whose
$E_{1}$-term is isomorphic to $[[K]]_{LR}$ with the first
differential $\partial_{LR}$, the $E_{2}$-term isomorphic to the
homology $H(K)_{LR}$, so that this spectral sequence converges to
the Lee-Rasmussen homology.
\end{thm}

Thus, two bigraded homology theories (the usual Khovanov theory with
height and quantum grading) and the one described above (with height
and dotted grading) both converge to the Lee-Rasmussen theory.

It is known that the Lee-Rasmussen theory give nice invariants
(quantum gradings of the two surviving elements). It would be
interesting to compare the convergence of the spectral sequence
describing above: {what is the meaning of the dotted grading of
surviving elements?}

\section{Applications}

The theory above has some obvious applications coming from the
definitions. Thus, if we work for knots in thickened surfaces, there
is a natural question whether such a knot can be destablised, i.e.,
some handles of the surface are nu\-ga\-to\-ry, or, in other words,
the representative of the knot given by this surface is minimal. The
surface $M$ has ${\bf Z}_{2}$-homology group of rank $k$, and if
they are all used as gradings of some homology groups of a knot in
$M\tilde\times I$, then the knot can not be destabilised.

\begin{crl}
If a set of {\em additional} gradings of non-trivial groups of
$Kh_{gg}(K)$ forms a subset in ${\bf R}^{k}$ not belonging to any
hypersurface passing through zero, then the link $K$ does not admit
destabilisation, i.e., there is no surface $M'$ of smaller genus
obtained from $M$ by a destabilisation so that the link $K$ lies in
the natural fibration over $M'$ generated by ${\cal M}\to M$.
\end{crl}

Analogously, the dotted grading can be used for estimating the
number of virtual crossings of a rigid virtual knot diagram.

Also, we mention (without any details, however) the facts which
generalise straightforwardly for the case of new gradings:

\begin{enumerate}

\item The homological length of the complex does not exceed the
number of classical crossings.

\item The spanning tree of Wehrli \cite{Weh} and Champanerkar-Kofman
\cite{ChK} saying that the Khovanov homology can be obtained from a
complex with a smaller chain group. This leads to the estimation for
the thickness:

 $Th(Kh(K))\le 2+g$, where $g$ is the genus of the atom
corresponding to the diagram $K$.

Here the thickness estimates the number of diagonals with slope $2$
on the plane with height and quantum gradings serving as
coordinates.

The same estimates can be obtained for our complex with new gradings
when looking at the {\em diagonals with respect to the former
gradings}. This leads to

\begin{thm}
For any knot $K$, the thickness of the dotted Khovanov homology
$Th(Kh(K))\le 2+g$, where $g$ is the genus of the atom.
\end{thm}

Together with the lemma saying that $span\langle K\rangle\le 4 n$,
where $n$ is the number of classical crossings, we get sharper
estimates for the number of crossings.

\item The Bar-Natan topological picture \cite{BN2} for
tangles and cobordisms, see also \cite{TuTu}. We need to generalize
Bar-Natan's topological category and construct a functor from it to
our category. We shall discuss this in a separate publication.

\item Rasmussen's estimates for the genus of a spanning surface;
here we must, indicate the category of cobordisms, say, for knots in
$M\times I$ we should consider spanning surfaces in $M\times I\times
I$.

\end{enumerate}

\section{The relation to other papers}

This paper generalises many constructions. First of all, we would
like to mention the work \cite{APS}, the work \cite{Frobenius} and
the work \cite{Izv}.

In fact, the idea of taking new gradings counting $X$ and $1$ on
non-trivial circles with opposite sides was originally used in
\cite{APS}. However, we used this approach for a more general
situation. For instance, the grading there was {\em necessary to
construct the Khovanov homology itself}; without it, the Khovanov
theory for knots in thickened surfaces does not exist; even with it,
it does not exist for knots in thickened ${\bf R}P^{2}$. We have
taken the approach from \cite{Izv} with twisted coefficient as the
basement for our homology theory (that allows us to give a fair
generalisation of Khovanov's theory for virtual and twisted knots
without any new gradings), and then introduced new gradings similar
to those ones by M.Asaeda, J.Przytycki and A.Sikora.

They used integral homology or even homotopy classes to define the
gradings. This was quite difficult for making it more algebraic.

We have axiomatized this approach taking the ${\bf Z}_2$-cohomology
(or just dotting) making it applicable to many other situations.

\end{document}